\documentclass[10pt]{article}
\usepackage{latexsym,amsfonts,amssymb,amsthm,amsmath,mathrsfs,color,amscd,graphicx,fullpage,fancyhdr,parskip,hyperref}
\usepackage[francais,english]{babel}

\newcommand{\s}[1]{{\mathcal #1}}

\newcommand{\bb}[1]{{\mathbb #1}}

\newtheorem{theorem}{Theorem}[section]
\newtheorem{corollary}[theorem]{Corollary}

\newtheorem{lemma}[theorem]{Lemma}
\newtheorem{proposition}[theorem]{Proposition}
\newtheorem{problem}[theorem]{Problem}
\newtheorem{definition}[theorem]{Definition}

\newtheorem{remark}[theorem]{Remark}

\numberwithin{equation}{section}
\numberwithin{theorem}{section}

\definecolor{Red}{cmyk}{0,1,1,0.2}
\def\dys{\displaystyle}
\def\ep{\epsilon}

\def\rg{\rangle} 
\def\lg{\langle} 
\def\ds{\displaystyle}
\def\T{\bb{T}}
\newcommand{\be}{\begin{equation}}
\newcommand{\ee}{\end{equation}}
\newcommand{\R}{\mathbb R}

\title{Mean field games systems of first order}
\author{Pierre Cardaliaguet\thanks{Ceremade, Universit\'e Paris-Dauphine,
Place du Maréchal de Lattre de Tassigny, 75775 Paris cedex 16 (France)} \and
P. Jameson Graber\thanks{Commands team (ENSTA ParisTech, INRIA Saclay), 828, Boulevard des Mar\'echaux, 91762 Palaiseau Cedex, Email: philip.graber@inria.fr}}

\begin{document}

\maketitle

\begin{abstract}
We consider a system of mean field games with local coupling in the deterministic limit.
Under general structure conditions on the Hamiltonian and coupling, we prove existence and uniqueness of the weak solution, characterizing this solution as  the minimizer of some optimal control of Hamilton-Jacobi and continuity equations.
We also prove that this solution converges in the long time average to the solution of the associated ergodic problem.

\end{abstract}


\section{Introduction} \label{sec:introduction}

Our purpose is to study the system
\begin{equation} \label{eq:mfg}\left\{
\begin{array}{cl}
(i) & -\partial_t\phi + H(x,D\phi) = f(x,m)  \\
(ii) & \partial_t m - \mathrm{div} \left(mD_p H(x,D\phi)\right) = 0 \\
(iii) & \phi(T,x) = \phi_T(x), m(0,x) = m_0(x).
\end{array}\right.
\end{equation}
System (\ref{eq:mfg}) is a model for first-order mean field games with local coupling.
Mean field games (MFG) were introduced simultaneously by Lasry and Lions \cite{lasry06,lasry06a,lasry07} and by Huang, Malham\'e, and Caines \cite{huang2006large,huang2007large} in order to study large population differential games.
The function $\phi$ in system (\ref{eq:mfg}) can be thought of as the value function for an average player seeking to optimize an objective functional,
while $m$ represents the time-evolving probability distribution of the state of the players.
The coupling between the two is represented here by the function $f(x,m)$.

The purpose of this article is to study the existence and uniqueness of weak solutions of the model (\ref{eq:mfg}) as well as their long time average behavior.
Structure conditions guaranteeing existence and uniqueness of solutions are already well-investigated in two general cases: the second order case with diffusion, and the first order case where the coupling is nonlocal and smoothing; see the discussion in Lasry and Lions \cite{lasry06a,lasry07} as well as the more recent contributions  by Gomes, Pimentel and S\'{a}nchez-Morgado \cite{GPSM1, GPSM2} and by Porretta \cite{Por13}.
Here the situation is one of first order equations with local coupling, about which much less is understood.
One approach, given in the lectures in \cite{lions07}, is to transform the system into a quasilinear elliptic equation in space time, thereby yielding smooth solutions.
However, this approach requires certain structure conditions (in particular to ensure that $m$ does not vanish) which we wish to abandon entirely.

The existence and uniqueness of weak solutions for this first order system under general structure conditions was studied in Cardaliaguet \cite{cardaliaguet2013weak} and in Graber  \cite{graber2013optimal}.
The approach, introduced by Benamou  and Brenier  \cite{bb} and carried on in Cardaliaguet, Carlier and Nazaret \cite{cardaliaguet2012geodesics} for optimal transport problems, was to characterize weak solutions in terms of minimizers for optimal control problems for some PDEs (Hamilton-Jacobi equations and transport equations).
We use a similar ideas in the present article, but we remove certain assumptions from \cite{cardaliaguet2013weak}.
In particular, the following two generalizations deserve emphasis:
\begin{itemize}
\item we completely dispense with hypothesis (H3) in \cite{cardaliaguet2013weak}, a strong restriction on the dependence of the Hamiltonian on space which would explicitly forbid, say, $H(x,D\phi) = c(x)|D\phi|^r$ for a positive continuous function $c$; for this we follow 
\cite{graber2013optimal}, where the analysis of mean field game system with local coupling and Hamiltonians positively homogeneous with respect to the gradient variable was performed;
\item we dispense with the growth assumption on $f(x,m)$ for $m$ near the origin; unlike in \cite{cardaliaguet2013weak, graber2013optimal}, we do not make the assumption $f(x,0) = 0$, and indeed we do not assume that $\lim_{m\to 0^+} f(x,m)$ is finite for all $x$.
\end{itemize}
Thus we allow for fairly general structure conditions, with the only major restriction being the relationship between the growth rate of the Hamiltonian and the coupling (Equation (\ref{eq:conqr})). Note that conditions linking the growth on the Hamiltonian to the growth of the coupling are fairly standard in mean field game theory (see, e.g., \cite{lasry07, GPSM1, GPSM2, Por13}). 
We prove the existence of solutions in an appropriately defined weak sense, characterizing the minimizers of two optimal control problems which are in duality (see Section \ref{sec:optimal_control}). 

Our second main result concerns the long time average of the solution of the mean field game system. Following standard arguments in control theory, one expects that, as horizon $T$ tends to infinity, the value function $\phi$ converges to the value of an ergodic control problem, while the measure $m$ stabilizes to an invariant measure. The resulting system should be therefore an ergodic MFG system, as introduced by Lasry and Lions in \cite{lasry06a}: 
$$
\left\{\begin{array}{cl}
(i)& \overline \lambda +H(x,D\overline \phi) =f(x,\overline m(x))\\
(ii) & -{\rm div} (\overline mD_pH(x,\overline D\phi))=0\\
(iii)& \overline m\geq 0, \; \int_{\T^d} \overline m= 1
\end{array}\right.
$$
This intuition turns out to be essentially correct, at least under suitable conditions. The first results in this direction were established by Lasry and Lions in \cite{lions07} and then  extended and sharpened by Cardaliaguet, Lasry, Lions, Porretta \cite{cardaliaguet2012long, cardaliaguet2013long2}: in these references, the authors are concerned with  second order mean field game systems, namely systems involving stochastic control problems with a nondegenerate diffusion. The main conclusion is that the rescaled map $(s,x)\to \phi(sT,x)/T$ converges to the (constant in space) map $s\to \bar \lambda(1-s)$ in $(0,1)\times \T^d$ while the (rescaled) map $(s,x)\to m(sT,x)$ converges to the (constant in time) map $x\to \bar m(x)$. Moreover  this last convergence holds at an exponential rate. Similar results in the discrete setting were also obtained by Gomes, Mohr and Souza in \cite{GMS10}. This long time average behavior is similar for first order MFG systems with nonlocal coupling  \cite{cardaliaguet2013long}, which correspond to systems of the form \eqref{MFG} in which the map $f$ is a nonlocal (and smoothing) function of the measure $m(t)$. The difficulty is then to define a notion of solution for the ergodic problem, since in this setting the measure $\bar m$ can have (and usually has) a singular part: this problem is overcome by using ideas from weak KAM theory.

In the present paper we address the same issue for the first order MFG systems with local coupling. The construction of solutions can be obtained as for the time-dependent problem. This construction was actually carried out by Evans in \cite{Ev} for the coupling $f(m)=\ln(m)$, where smoothness of the solution is established as well. Under our general assumptions, we cannot hope to obtain smooth solutions, but we show the existence and uniqueness of weak solutions. 

To prove the convergence of the time dependent system \eqref{eq:mfg} to the ergodic one, we  extend a classical energy inequality introduced by Lasry and Lions \cite{lasry07}  to our weak solutions (see Proposition \ref{prop:energie}): this provides the convergence of $m$. The convergence of $\phi/T$ is more subtle. Indeed, in \cite{cardaliaguet2012long, cardaliaguet2013long2}, it relies on the fact that $m$ does not vanish for second order MFG systems; in \cite{cardaliaguet2013long}, we used the comparison  principle for Hamilton-Jacobi equations (HJ equations). In the present context, the measure $m$ is expected to vanish and the notion of solution is too weak to allow for a comparison argument.
To overcome this issue, we  pass to the limit in the optimization problem that characterizes weak solutions of \eqref{eq:mfg}.


The outline of the paper is as follows.
In the following subsection, we list the basic assumptions which hold for all of our main results.
Then in Section \ref{sec:optimal_control} we present two optimal control problems, one for the Hamilton Jacobi equation and the other for the continuity equation, whose minimizers are characterized by weak solutions to the mean field games system (\ref{eq:mfg}).
Section \ref{sec:weak_solutions} is devoted to the study of weak solutions, in particular their existence and uniqueness.
In Section \ref{sec:2opti}, we study weak solutions for the corresponding {\em ergodic problem.}
Finally, in Section \ref{sec:asymp}, we study the link between the time-dependent and ergodic problems.

\bigskip
{\bf Acknowledgement: }  
This work has been partially supported by the Commission of the
European Communities under the 7-th Framework Programme Marie
Curie Initial Training Networks   Project SADCO,
FP7-PEOPLE-2010-ITN, No 264735, and by the French National Research Agency
 ANR-10-BLAN 0112 and ANR-12-BS01-0008-01.

\subsection{Notation and assumptions} 
\label{sec:assumptions}

{\bf Notation:} We denote by 
 $\lg x, y\rg$ the Euclidean scalar product  of two vectors $x,y\in\R^d$ and by
$|x|$ the Euclidean norm of $x$.  We work in the flat $d-$dimensional torus $\bb{T}^d=\R^d\backslash \mathbb Z^d$. For $k,n\in \mathbb N$ and $T>0$, we denote by ${C}^k([0,T]\times \bb{T}^d, \R^n)$ the space of maps $\phi=\phi(t,x)$ of class ${C}^k$ in time and space with values in $\R^n$. For $p\in [1,\infty]$ and $T>0$, we denote by $L^p(\bb{T}^d)$ and $L^p((0,T)\times \bb{T}^d)$ the set of $p-$integrable maps over $\bb{T}^d$ and $[0,T]\times \bb{T}^d$ respectively. We often abbreviate $L^p(\bb{T}^d)$ and $L^p((0,T)\times \bb{T}^d)$ into  $L^p$. We denote by $\|f\|_p$ the $L^p-$norm of a map $f\in L^p$. For $f\in L^1((0,1)\times \bb{T}^d)$, we define $\lg f(t)\rg$ to be the (a.e. defined) quantity $\int_{\bb{T}^d}f(t,x)dx$.

If $\mu$ is a vector measure over $\bb{T}^d$ or $[0,T]\times \bb{T}^d$, we denote by $\mu^{ac}$ and $\mu^s$ the decomposition of $\mu$ in absolutely continuous part and singular part with respect to the Lebesgue measure. Recall that $\mu=\mu^{ac}+\mu^s$. For simplicity, if $\phi\in BV$ over $[0,T]\times \bb{T}^d$,  we abbreviate the notation $(\partial_t\phi)^{ac}$  into  $\partial_t\phi^{ac}$. \\

{\bf Assumption: } We now list the various conditions needed on the data of the problem. These assumptions are in force throughout the paper. 

\begin{enumerate}
\item (Conditions on the initial and final conditions) $m_0$ is a probability measure on $\bb{T}^d$ which is absolutely continuous with respect to Lebesgue measure, having density which we also call $m_0$ in $C(\bb{T}^d)$. We suppose moreover that $m_0>0$ on $\bb{T}^d$. 
We  assume that $\phi_T:\T^d\to \R$ is a Lipschitz continuous function on $\bb{T}^d$.
\item (Conditions on the Hamiltonian) $H:\bb{T}^d \times \bb{R}^d \to \bb{R}$ is continuous in both variables, convex and differentiable in the second variable, with $D_pH$ continuous in both variables. Moreover, $H$ has superlinear growth in the gradient variable: there exist $r >1 $ and $C >0$ such that
\begin{equation}
\label{eq:hamiltonian_bounds}
\frac{1}{rC}|p|^r-C \leq H(x,p) \leq \frac{C}{r}|p|^r + C,
\end{equation}
We denote by $H^*(x,\cdot)$ the Fenchel conjugate of $H(x,\cdot)$, which, due to the above assumptions, satisfies
\begin{equation}
\label{eq:hamiltonian_conjugate_bounds}
\frac{1}{r'C}|q|^{r'}-C \leq H^*(x,q) \leq \frac{C}{r'}|q|^{r'} + C,
\end{equation}
where $r'$ is the conjugate of $r$. 
We will also denote by $L$ the Lagrangian given by $L(x,q) = H^*(x,-q)$, which thus satisfies the same bounds as $H^*$.
\item (Conditions on the coupling) Let $f$ be continuous on $\bb{T}^d \times (0,\infty)$, strictly increasing in the second variable, satisfying
\begin{equation} \label{eq:coupling_growth}
\frac{1}{C}|m|^{q-1} - C \leq f(x,m) \leq C|m|^{q-1} + C ~~ \forall ~ m \geq 1.
\end{equation}
\item The  relation holds between the growth rates of $H$ and of $F$: 
\be\label{eq:conqr}
r > \max\{d(q-1),1\}.
\ee 
\end{enumerate}
Note that condition \eqref{eq:conqr} implies that the growth of $f$ (of order $q-1$) has to be much smaller than the growth of $H$ (of order $r$). 

We define $F$ so that $F(x,\cdot)$ is a primitive of $f(x,\cdot)$ on $(0,\infty)$, that is,
\begin{equation}
F(x,m) = \int_1^m f(x,s)ds, ~~ \forall ~ m > 0.
\end{equation}
It follows that $F$ is continuous on $\bb{T}^d \times (0,\infty)$, is strictly convex and differentiable in the second variable, and satisfies the growth condition
\begin{equation} \label{eq:cost_growth}
\frac{1}{qC}|m|^q - C \leq F(x,m) \leq \frac{C}{q}|m|^q + C~~~ \forall ~ m \geq 1.
\end{equation}
For $m < 0$ we set $F(x,m) = +\infty$. 
We denote by $F(x,0)$ the limit $\lim_{m \to 0^+} F(x,m)$, which may be finite or $+\infty$ (see Remark \ref{rem:F} below).

We will denote throughout the conjugate exponent of $q$ by $p = q^*$. Define $F^*(x,\cdot)$ to be the Fenchel conjugate of $F(x,\cdot)$ for each $x$. Note that
\begin{equation} \label{eq:cost_growth_star}
\frac{1}{pC}|a|^p - C \leq F^*(x,a) \leq \frac{C}{p}|a|^p + C, ~~ \forall a \geq 0.
\end{equation}

\begin{remark}\label{rem:F}{\rm Note that our assumptions imply that $F$ is bounded below on $\bb{T}^d\times (0,+\infty)$: indeed, by \eqref{eq:cost_growth},  $F$ is uniformly coercive at infinity, so that we only have to worry about $m\in (0,1)$. For such $m$'s, we have by convexity of $F$: 
$$
F(x,m)\geq F(x,1)+f(x,1)(m-1)\geq \min_{\bb{T}^d}F(\cdot,1) -\max_{\bb{T}^d}f(\cdot,1).
$$
Finally, note that, since $F(x,m)=+\infty$ for $m<0$, $F^*(x,a)= \sup_{m\geq 0} ma -F(x,m)$ is nondecreasing. 
}\end{remark} \medskip

\begin{remark}{\rm 
Our assumptions do not prohibit the possibility $\lim_{m \to 0^+} f(x,m) = -\infty$,
which creates special difficulties in proving the existence of solutions, as we will see in the proof of Theorem \ref{thm:existence}.
By way of comparison, let us review what happens if we posit that $\lim_{m \to 0^+} f(x,m)$ is finite.
Then without loss of generality we may suppose it is zero (otherwise change $f$ by a continuous function depending only on $x$, which does not change the assumptions on its growth in $m$).
In this case, as we analyze the optimal control of the Hamilton-Jacobi equation (see Section \ref{sec:optimal_control} below), the right-hand side can be assumed positive, which turns out to be a boon for regularity of the solution.
That is, the solution turns out to be H\"{o}lder continuous thanks to nothing more than the $L^q$ regularity of the right-hand side (the ``control").
Such a result has been proved in \cite{cardaliaguet2012holder} and used in \cite{cardaliaguet2013weak} to construct solutions of systems of the form (\ref{eq:mfg}) for which $\phi$ is a continuous viscosity solution to the Hamilton-Jacobi equation.
Since we no longer have this assumption, we will be forced to look for solutions with lower regularity.
}\end{remark}

\section{Optimal control problems}
\label{sec:optimal_control}

Throughout this section we study several optimal control problems: we will see in the next section that the MFG system \eqref{eq:mfg} is the system of optimality conditions for these problems. 

\subsection{Optimal control of HJ equations}

Denote by $\s{K}_0$ the set of maps $\phi \in C^1([0,T] \times \bb{T}^d)$ such that $\phi(T,x) = \phi_T(x)$. Then define on $\s{K}_0$ the functional
\begin{equation}
\s{A}(\phi) = \int_0^T \int_{\bb{T}^d} F^*(x,-\partial_t \phi + H(x,D\phi)) dx dt - \int_{\bb{T}^d} \phi(0,x) dm_0(x) .
\end{equation}
Then we have our first optimal control problem. \medskip
\begin{problem}[Optimal control of HJ] \label{pr:smooth}
Find $\inf_{\phi \in \s{K}_0} \s{A}(\phi)$.
\end{problem} \medskip

We now look at the dual problem. Define $\s{K}_1$ to be the set of all pairs $(m,{\bf w}) \in L^1((0,T) \times \bb{T}^d) \times L^1((0,T) \times \bb{T}^d;\bb{R}^d)$ such that $m \geq 0$ almost everywhere, $\int_{\bb{T}^d}m(t,x)dx = 1$ for a.e. $t \in (0,T)$, and
\begin{align*}
\partial_t m + \mathrm{div}~(w) &= 0\\
m(0,\cdot) &= m_0(\cdot)
\end{align*}
in the sense of distributions. Because of the integrability assumption on $\bf w$, it follows that $t \mapsto m(t)$ has a unique representative such that $\int m(t)\phi$ is continuous on $[0,T]$ for all $\phi \in C^\infty(\bb{T}^d)$ (cf. \cite{ambrosio08}). It is to this representative that we refer when we write $m(t)$, and thus $m(t)$ is well-defined for all $t \in [0,T]$.

Define a functional
\begin{equation}\label{eq:dual}
\s{B}(m,w) = \int_{\bb{T}^d} \phi_T(x)m(T,x)dx + \int_0^T \int_{\bb{T}^d} m(t,x)L\left(x,\frac{w(t,x)}{m(t,x)}\right) + F(x,m(t,x)) dx dt
\end{equation}
on $\s{K}_1$. Recall that $L$ is defined juste after \eqref{eq:hamiltonian_conjugate_bounds}.  We follow the convention that
\begin{equation}\label{conventionH*}
mH^*\left(x,-\frac{w}{m}\right) = \left\{ \begin{array}{ll} \infty & \mbox{\rm if $m=0$ and $w\neq 0$} \\ 0 & \mbox{\rm if $m=0$ and $w = 0$} \end{array} \right.
\end{equation}
Since $m \geq 0$, the second integral in (\ref{eq:dual}) is well-defined in $(-\infty,\infty]$ by the assumptions on $F$ and $L$ (see also Remark \ref{rem:F}). The first integral is well-defined and necessarily finite by the continuity of $\phi_T$ and the fact that $m(T,x)dx$ is a probability measure.

We next state the ``dual problem" succinctly as \medskip
\begin{problem}[Dual Problem] \label{pr:dual}
Find $\inf_{(m,w) \in \s{K}_1} \s{B}(m,w)$.
\end{problem} \medskip

The main result of this section is \medskip
\begin{theorem}
\label{thm:dual} Problems \ref{pr:smooth} and \ref{pr:dual} are in duality, i.e.
\begin{equation} \label{eq:duality}
\inf_{\phi \in \s{K}_0} \s{A}(\phi) = -\min_{(m,w) \in \s{K}_1} \s{B}(m,w)
\end{equation}
Moreover, the minimum on the right-hand side is achieved by a unique pair $(m,w) \in \s{K}_1$ which must satisfy $m \in L^q((0,T) \times \bb{T}^d)$ and $w \in L^{\frac{r'q}{r'+q-1}}((0,T) \times \bb{T}^d)$.
\end{theorem} \medskip


\begin{proof}[Proof of Theorem \ref{thm:dual}]
Proving Equation (\ref{eq:duality}) is simply an application of the Fenchel-Rockafellar theorem (see \cite{ekeland1976convex}) as in \cite{cardaliaguet2012geodesics} and \cite{cardaliaguet2013weak}. To see the rest, we suppose that $(m,w) \in \s{K}_1$ minimizes $\s{B}$. Since $L$ satisfies the growth conditions (\ref{eq:hamiltonian_conjugate_bounds}) we have
\begin{equation}
\int_0^T \int_{\bb{T}^d} \frac{1}{C}\frac{|w|^{r'}}{m^{r'-1}} - C ds dt \leq \int_0^T \int_{\bb{T}^d} mL\left(x,\frac{w}{m}\right) ds dx.
\end{equation}
On the other hand, since $F(x,\cdot)$ is bounded below (see Remark \ref{rem:F}),
\begin{equation}
\s{B}(m,w) \geq -C + \int_0^T \int_{\bb{T}^d} mL\left(x,\frac{w}{m}\right) ds dx
\end{equation}
for some $C > 0$. Since $\s{B}(m,w)$ must be finite, it follows that $\displaystyle \frac{|w|^{r'}}{m^{r'-1}} \in L^1$. Now we use the growth conditions (\ref{eq:cost_growth}) on $F$ to deduce
\begin{equation}
\int_0^T \int_{\bb{T}^d} |m|^q dx dt \leq C(T+1) + C\int_0^T \int_{\bb{T}^d} F(x,m) dx dt.
\end{equation}
Again appealing to (\ref{eq:cost_growth_star}) we have that
\begin{equation}
\s{B}(m,w) \geq - C + \int_0^T \int_{\bb{T}^d} F(x,m) ds dx \geq - C + \frac{1}{C}\int_0^T \int_{\bb{T}^d} |m|^q dx dt
\end{equation}
for large enough $C$. Hence $m \in L^q$. Using H\"{o}lder's inequality we get that $w \in L^{\frac{r'q}{r'+q-1}}((0,T) \times \bb{T}^d)$. Uniqueness comes from the fact that $F(x,\cdot)$ and $L(x,\cdot)$ are strictly convex.

\end{proof}

\subsection{Relaxation of Problem \ref{pr:smooth}}
\label{sec:relaxed}

We do not expect Problem \ref{pr:smooth} to have a solution. For this reason we introduce a new problem on a wider class of functions $\s{K}$ and show in Proposition \ref{prop:relaxed} that it is the relaxation of Problem \ref{pr:smooth}.  \medskip

\begin{definition}[Relaxed set] \label{def:relaxed_space}
The set $\s{K}$ will be defined as the set of all pairs $(\phi,\alpha) \in BV \times L^1$ such that $D\phi \in L^r((0,T) \times \bb{T}^d)$,
$\phi(T,\cdot) \leq  \phi_T$ in the sense of traces, $\alpha_+ \in L^p((0,T) \times \bb{T}^d)$, $\phi \in L^\infty((t,T)\times\bb{T}^d)$ for every $t \in (0,T)$, and
\begin{equation} \label{eq:inequality_in_distribution}
-\partial_t \phi + H(x,D\phi) \leq \alpha.
\end{equation}
in the sense of distribution.
\end{definition} \medskip
We remark that $\s{K}$ is a convex subset of the space $BV \times L^1$, and that for $(\phi,\alpha) \in \s{K}$, $H(x,D\phi) \in L^1$ by the bounds in (\ref{eq:hamiltonian_bounds}). \medskip  
\begin{problem}[Relaxed Problem] \label{pr:relaxed}
Find 
$$\inf_{(\phi,\alpha) \in \s{K}} \s{A}(\phi,\alpha) = \inf_{(\phi,\alpha) \in \s{K}} \int_0^T \int_{\bb{T}^d} F^*(x,\alpha(t,x))dx dt - \int_{\bb{T}^d} \phi(0,x)m_0(x)dx.$$
\end{problem} \medskip

Note that $\s{A}(\phi,\alpha)$ is well-defined in $(-\infty,+\infty]$ for any $(\phi,\alpha) \in \s{K}$: indeed, $\phi(0, \cdot)$ is well defined in $L^1$ in the sense of trace and, if we set $K_1:= \max_{\bb{T}^d} F(\cdot,1)$, then we have $F^*(x,a)\geq a-K_1$, so that, as $\alpha\in L^1$,  
$$
 \int_0^T \int_{\bb{T}^d} F^*(x,\alpha(t,x))dx dt \geq  \int_0^T \int_{\bb{T}^d} \alpha(t,x)dx dt- TK_1\ >\ -\infty. 
$$
The main result of this section is the following: \medskip

\begin{proposition} \label{prop:relaxed}
The relaxed problem \ref{pr:relaxed} is in duality with \ref{pr:dual}, i.e. \begin{equation}
\inf_{(\phi,\alpha) \in \s{K}} \s{A}(\phi,\alpha) = - \min_{(m,w) \in \s{K}_1} \s{B}(m,w).
\end{equation}
Equivalently, the infimum appearing in Problem \ref{pr:relaxed} is the same as that appearing in Problem \ref{pr:smooth}.
\end{proposition} \medskip

Before proving the proposition, let us start with a key consequence of Equation \eqref{eq:inequality_in_distribution}: \medskip

\begin{lemma} \label{lem:integration_by_parts}
Suppose $(\phi,\alpha) \in \s{K}$. If $(m,w) \in \s{K}_1$  is such that $m \in L^q$ and $mL(\cdot,\frac{w}{m}) \in L^1$, then $\alpha m$ is integrable and we have
\begin{equation} \label{eq:integration_by_parts}
\int_{\bb{T}^d} \phi(t,x)m(t,x)dx \leq \int_{\bb{T}^d} \phi(s,x)m(s,x)dx + \int_t^s \int_{\bb{T}^d} m(\tau,x)L\left(x,\frac{w(\tau,x)}{m(\tau,x)}\right) + \alpha(\tau,x)m(\tau,x)dx d\tau,
\end{equation} 
for almost every $0 < t < s < T$. Moreover, 
\begin{align}\label{eq:integration_by_parts2}
\int_{\bb{T}^d} \phi(0,x)m_0(x)dx &\leq \int_{\bb{T}^d} \phi(t,x)m(t,x)dx + \int_0^t \int_{\bb{T}^d} m(\tau,x)L\left(x,\frac{w(\tau,x)}{m(\tau,x)}\right) + \alpha(\tau,x)m(\tau,x)dx d\tau,\\ \label{eq:integration_by_parts3}
\int_{\bb{T}^d} \phi(t,x)m(t,x)dx &\leq \int_{\bb{T}^d} \phi_T(x)m(T,x)dx + \int_t^T \int_{\bb{T}^d} m(\tau,x)L\left(x,\frac{w(\tau,x)}{m(\tau,x)}\right) + \alpha(\tau,x)m(\tau,x)dx d\tau
\end{align}
for almost every $t \in (0,T)$.
\end{lemma}

\begin{proof} We first assume that $\alpha\in L^p$ and remove this assumption at the end of the proof. 
We first extend $(m,w)$ to $\bb{R}\times \bb{T}^d$ by setting $(m,w)=(m_0,0)$ on $(-\infty,0)\times \bb{T}^d$ and $(m,w)= (m(T),0)$ on $(T, +\infty)\times \bb{T}^d$. Note that we still have $\partial_t m+ {\mathrm{div}}~w = 0$ on $\bb{R}\times \bb{T}^d$.
Let $\xi$ be a standard convolution kernel on $\R\times \R^d$ such that $\xi > 0$ and let $\xi_\epsilon(t,x) := \xi((t,x)/\epsilon)/\epsilon^{d+1}$. Define $m_\epsilon := \xi_\epsilon \ast m$ and $w_\epsilon := \xi_\epsilon \ast w$. Then $m_\epsilon$ and $w_\epsilon$ are $C^\infty$ smooth, $m_\epsilon > 0$, and
\begin{equation}
\partial_t m_\epsilon + {\mathrm{div}}~w_\epsilon = 0.
\end{equation}
Fix $t, s \in (0,T)$ with $t<s$. Integrating against $\phi$ over $[t,s] \times \bb{T}^d$ we have
\begin{equation}
\int_{\bb{T}^d} \phi(s)m_\epsilon(s) - \phi(t)m_\epsilon(t) - \int_t^s \int_{\bb{T}^d} \partial_t \phi m_\epsilon + (D\phi,w_\epsilon) = 0.
\end{equation}
Here $\phi(s)$ and $\phi(t)$ are functions in $L^1(\bb{T}^d)$ based on trace theory, $\partial_t \phi$ is a signed Radon measure, and we recall that $D\phi$ is integrable by virtue of $(\phi,\alpha)$ belonging to $\s{K}$.
Note that
\begin{equation}
- \int_t^s \int_{\bb{T}^d} (D\phi,w_\epsilon) \leq  \int_t^s \int_{\bb{T}^d} H(x,D\phi)m_\epsilon + m_\epsilon L\left(x,\frac{w_\epsilon}{m_\epsilon}\right).
\end{equation}
Recalling that $-\partial_t \phi + H(x,D\phi) \leq \alpha$ in distribution, we deduce
\begin{equation}
\int_{\bb{T}^d} \phi(t)m_\epsilon(t) \leq \int_{\bb{T}^d} \phi(s)m_\epsilon(s) + \int_t^s \int_{\bb{T}^d} m_\epsilon L\left(x,\frac{w_\epsilon}{m_\epsilon}\right) + m_\epsilon \alpha.
\end{equation}
As $\epsilon \to 0$, we have that $m_\epsilon \to m$ in $L^q((0,T) \times \bb{T}^d)$, and in particular $m_\epsilon(\tau) \to m(\tau)$ in $L^q(\bb{T}^d)$ for almost every $\tau \in (0,T)$, while $m_\epsilon \alpha \to m \alpha$ in $L^1$ since $\alpha \in L^p$. Thus as $\phi(\tau) \in L^p(\bb{T}^d)$ for almost every $\tau \in (0,T)$, we get $\int\phi(\tau)m_\epsilon(\tau) \to \int \phi(\tau)m(\tau)$ for almost every $\tau \in (0,T)$. We assume in particular that this holds for $\tau = t,s$. Then
\begin{equation}
\int_{\bb{T}^d} \phi(t)m(t) \leq \int_{\bb{T}^d} \phi(s)m(s) + \limsup_{\epsilon \to 0} \int_t^s \int_{\bb{T}^d} m_\epsilon L\left(x,\frac{w_\epsilon}{m_\epsilon}\right) + m \alpha.
\end{equation}

To conclude, we just need to show  that
\begin{equation} \label{eq:limit_L}
\lim_{\epsilon \to 0}  \int_t^s \int_{\bb{T}^d} m_\epsilon L\left(x,\frac{w_\epsilon}{m_\epsilon}\right) =  \int_t^s \int_{\bb{T}^d} m L\left(x,\frac{w}{m}\right).
\end{equation}
Since the map $(m,w)\to mL(x, w/m)$ is lower semicontinuous and bounded below, Fatou Lemma implies 
$$
\liminf_{\epsilon \to 0}  \int_t^s \int_{\bb{T}^d} m_\epsilon L\left(x,\frac{w_\epsilon}{m_\epsilon}\right) \geq  \int_t^s \int_{\bb{T}^d} m L\left(x,\frac{w}{m}\right).
$$
 The difficult (and useful) inequality is the opposite one. We first note that 
\be\label{hjbkgdv}
 \int_t^s \int_{\bb{T}^d} \xi_\ep \ast \left(m L\left(\cdot,\frac{w}{m}\right)\right)(x)dxdr
 = 
 \int_t^s \int_{\bb{T}^d} \xi_\ep \ast \left(m L\left(x,\frac{w}{m}\right)\right)(x)dxdr
 +R_\ep
 \ee
 where $\ds  R_\ep=  \int_t^s \int_{\bb{T}^d} \zeta_\ep(r,x)dxdr, $ with
$$
\zeta_\ep(r,x):= \int_\R\int_{\bb{\R^d}} \xi_\ep(r-\tau, x-y)\left(m(\tau,y)L\left(y,\frac{ w(\tau,y)}{m(\tau,y)}\right)
 -m(\tau,y)L\left(x,\frac{w(\tau,y)}{m(\tau,y)}\right) \right)dyd\tau
 $$
As $\ep\to 0$, the left-hand side of \eqref{hjbkgdv} converges to $\ds  \int_t^s \int_{\bb{T}^d} m L\left(\cdot,\frac{w}{m}\right)$. On another hand, by the convexity of the map $(m,w)\to mL(x, w/m)$, we have 
 $$
 \int_t^s \int_{\bb{T}^d} \xi_\ep\ast\left(m L\left(x,\frac{w}{m}\right)\right)(x)dxdr
\geq 
 \int_t^s \int_{\bb{T}^d} m_\ep L\left(x,\frac{w_\ep}{m_\ep}\right).
 $$
 So we are left to show that $R_\ep\to 0$. Note that $\zeta_\ep\to 0$ a.e.  The bounds on $L$ indicate
$$
|\zeta_\ep(r,x)| \leq 2C \int_\R\int_{\bb{\R^d}} \xi_\ep(r-\tau, x-y)\left(\frac{|w(\tau,y)|^{r'}}{m^{r'-1}(\tau,y)}+1\right)dy d\tau
 $$
where the right-hand side converges in $L^1$ to $2C \frac{|w|^{r'}}{m^{r'-1}}+1$ since this map is in $L^1$  by the proof of Theorem \ref{thm:dual}. 
 By the Dominate Convergence Theorem we conclude that $R_\ep\to 0$ and \eqref{eq:limit_L} holds. 
 
We now want to extend this argument to the case where $t = 0$ or $s = T$. For the case, $s = T$, note that
$$
\int_{\bb{T}^d} \phi(t)m_\epsilon(t) \leq \int_{\bb{T}^d} \phi_Tm_\epsilon(T) + \int_t^T \int_{\bb{T}^d} m_\epsilon L\left(x,\frac{w_\epsilon}{m_\epsilon}\right) + m_\epsilon \alpha.
$$
It suffices to show that
\begin{equation}
 \int_{\bb{T}^d} \phi_T m_{\epsilon}(T) \to \int_{\bb{T}^d} \phi_T m(T), ~~ \epsilon \to 0,
\end{equation}
and the rest of the argument follows as before.
We can choose a sequence $s_n \to T$ such that for a fixed $n$, $m_\epsilon(s_n) \to m(s_n)$ in $L^q(\bb{T}^d)$ as $\epsilon \to 0$. Then recall that $m_\epsilon$ and $m$ are both weakly continuous in time, and in particular $t \mapsto \int \phi_T m(t)$ is continuous. Since $m_\epsilon$ is a regularization of $m$ by convolution, it follows that $t \mapsto \int \phi_T m_\epsilon(t)$ is continuous with modulus of continuity independent of $\epsilon$. Let us examine the three terms on the right-hand side of the following estimate:
\begin{multline*}
\left|\int_{\bb{T}^d} \phi_T m_{\epsilon}(T) - \int_{\bb{T}^d} \phi_T m(T) \right| \leq \left|\int_{\bb{T}^d} \phi_T m_{\epsilon}(T) - \int_{\bb{T}^d} \phi_T m_\epsilon(s_n) \right| \\
+ \left|\int_{\bb{T}^d} \phi_T m_{\epsilon}(s_n) - \int_{\bb{T}^d} \phi_T m(s_n) \right| + \left|\int_{\bb{T}^d} \phi_T m(s_n) - \int_{\bb{T}^d} \phi_T m(T) \right|.
\end{multline*}
The first and third terms can be made small by choosing $n$ large enough, independent of $\epsilon$. The second term can then be made small be choosing $\epsilon$ small, which yields the desired result.

Now we consider the case $t = 0$.
To do this we are going to choose a different smooth approximation of $m$ and $w$. Let $\eta : \bb{R} \to (0,\infty)$ and $\psi : \bb{R}^d \to (0,\infty)$ be convolution kernels, with $\eta_\epsilon(t) = \epsilon^{-1}\eta(t/\epsilon)$ and $\psi_\delta(x) = \delta^{-d}\psi(x/\delta)$. Let $\xi_{\epsilon,\delta}(t,x) = \eta_\epsilon(t)\psi_\delta(x)$ and define $m_{\epsilon,\delta} := \xi_{\epsilon,\delta} \ast m$ and $w_{\epsilon,\delta} := \xi_{\epsilon,\delta} \ast w$. We have that
$$
\partial_t m_{\epsilon,\delta} + \mathrm{div}~w_{\epsilon,\delta} = 0,
$$
and so we obtain, as before,
$$
\int_{\bb{T}^d} \phi(0)m_{\epsilon,\delta}(0) \leq \int_{\bb{T}^d} \phi(t)m_{\epsilon,\delta}(t) + \int_0^t \int_{\bb{T}^d} m_{\epsilon,\delta} L\left(x,\frac{w_{\epsilon,\delta}}{m_{\epsilon,\delta}}\right) + m_{\epsilon,\delta} \alpha.
$$
As $\epsilon,\delta$ both tend to zero, we have that $m_{\epsilon,\delta} \to m, w_{\epsilon,\delta} \to w$ in $L^q$. So the same arguments as above hold, if we can just show that
\begin{equation}
\int_{\bb{T}^d} \phi(0)m_{\epsilon,\delta}(0) \to \int_{\bb{T}^d} \phi(0)m_0.
\end{equation}
on some subsequence $\epsilon,\delta \to 0$. Our first observation is that
\begin{align*}
|m_{\epsilon,\delta}(0,x) - m_0(x)| &\leq \left|\iint \eta_\epsilon(-s)\psi_\delta(x-y)[m(s,y)-m_0(y)]dy ds \right|\\
&~~~~~~~~~ + \left|\iint \eta_\epsilon(-s)\psi_\delta(x-y)[m_0(y)-m_0(x)]dy ds \right|\\
&\leq \left|\iint \eta_\epsilon(-s)\psi_\delta(x-y)[m(s,y)-m_0(y)]dy ds \right| + \omega(\delta),
\end{align*}
where $\omega$ is the modulus of continuity of $m_0$. Next, we observe that
\begin{align*}
\left|\iint \eta_\epsilon(-s)\psi_\delta(x-y)[m(s,y)-m_0(y)]dy ds \right| &= \left|\iint \eta_\epsilon(-s)\psi_\delta(x-y)\int_0^s \partial_t m(\tau,y) d\tau dy ds \right|\\
&= \left|\int_{-\epsilon}^\epsilon \int_{\bb{T}^d} \int_0^s \eta_\epsilon(-s)\psi_\delta(x-y) \mathrm{div}~ w(\tau,y) d\tau dy ds \right|\\
&\leq \left|\int_{-\epsilon}^\epsilon \int_{\bb{T}^d} \int_0^s \eta_\epsilon(-s) D \psi_\delta(x-y) w(\tau,y) d\tau dy ds \right|\\
&\leq \frac{1}{\delta^{2d}}\int_{-\epsilon}^\epsilon \int_{\bb{T}^d} \int_0^\epsilon \eta_\epsilon(-s) \left|D \psi\left(\frac{x-y}{\delta}\right)\right| |w(\tau,y)| d\tau dy ds \\
&\leq \frac{C}{\delta^{2d}} \int_0^\epsilon \int_{\bb{T}^d} |w(\tau,y)| d\tau dy.
\end{align*}
Since $w$ is integrable, we can take $\epsilon = \epsilon(\delta)$ small enough relative to $\delta$ to get $m_{\epsilon(\delta),\delta}(0) \to m_0$ uniformly as $\delta \to 0$. It follows that $\int \phi(0)m_{\epsilon,\delta}(0) \to \int \phi(0)m_0$, and the rest of the argument proceeds as above.

It now remains to remove the assumption $\alpha \in L^p$. For $M>0$ large, let us set $\alpha_M= \alpha\vee (-M)$. Then the pair $(\phi,\alpha_M)$ belongs to $\s{K}$ and $\alpha_M\in L^p$ since $\alpha_+\in L^p$. So we can apply the above results to this pair. In particular,  
$$
\int_{\bb{T}^d} \phi(0,x)m_0(x)dx \leq \int_{\bb{T}^d} \phi_T(x)m(T,x)dx + \int_0^T \int_{\bb{T}^d} m(\tau,x)L\left(x,\frac{w(\tau,x)}{m(\tau,x)}\right) + \alpha_M(\tau,x)m(\tau,x)dx d\tau.
$$
Hence $\displaystyle \int_0^T \int_{\bb{T}^d} (\alpha_-\wedge M)m\leq C$ for any $M$, which shows that $(\alpha_-)m$ is integrable. Then $\alpha m$ is also integrable and we can easily complete the proof of \eqref{eq:integration_by_parts}, \eqref{eq:integration_by_parts2} and \eqref{eq:integration_by_parts3} by approximation. 
\end{proof}

\begin{proof}[Proof of Proposition \ref{prop:relaxed}.]
Since for $\phi \in \s{K}_0$ we have that $(\phi,-\partial_t \phi + H(x,D\phi)) \in \s{K}$, it follows that\\ $\inf_{(\phi,\alpha) \in \s{K}} \s{A}(\phi,\alpha) \leq \inf_{\phi \in \s{K}_0} \s{A}(\phi)$. It therefore suffices to prove the opposite inequality. Equivalently, let $(m,w)$ be a minimizer of the dual problem \ref{pr:dual}; it suffices to show by Theorem \ref{pr:dual} that for $(\phi,\alpha) \in \s{K}$, we have $\s{A}(\phi,\alpha) \geq -\s{B}(m,w)$. As we will see below, this essentially follows from Lemma \ref{lem:integration_by_parts}.

From Lemma \ref{lem:integration_by_parts}, $\alpha m$ is integrable and we have
\begin{equation*}
\int_0^T \int_{\bb{T}^d} F^*(x,\alpha(t,x)) + F(x,m(t,x))dx dt \geq \iint \alpha m
\end{equation*}
and 
\begin{equation*}
\iint \alpha m \geq -\iint mL\left(x,\frac{w}{m}\right) + \int_{\bb{T}^d} \phi(0)m_0 - \phi_T m(T).
\end{equation*}
Putting the above estimates together,
\begin{align*}
\s{A}(\phi,\alpha) &= \int_0^T \int_{\bb{T}^d} F^*(x,\alpha(t,x))dx dt - \int_{\bb{T}^d} \phi(0,x)m_0(x)dx\\
&\geq -\int_0^T \int_{\bb{T}^d} F(x,m(t,x)) + m(t,x)L\left(x,\frac{w(t,x)}{m(t,x)}\right)dx dt - \int_{\bb{T}^d} \phi_T(x)m(T,x)dx\\
&=  - \s{B}(m,w).
\end{align*}
\end{proof}

We complete this part by a simple remark concerning the case of equality in \eqref{eq:integration_by_parts}: \medskip

\begin{corollary}\label{cor:egalite} Assume that $(\phi,\alpha)$ and $(m, w)$ are as in Lemma \ref{lem:integration_by_parts} and that 
\begin{equation}\label{lkqhsdn}
\int_{\bb{T}^d} \phi(0,x)m_0(x)dx = \int_{\bb{T}^d} \phi_T(x)m(T,x)dx + \int_0^T \int_{\bb{T}^d} m(\tau,x)L\left(x,\frac{w(\tau,x)}{m(\tau,x)}\right) + \alpha(\tau,x)m(\tau,x)dx d\tau.
\end{equation}
Then $-\partial_t\phi^{ac}(t,x)+H(x,D\phi(t,x))= \alpha(t,x)$ for $m-$a.e. $(t,x)\in (0,T)\times \bb{T}^d$, where $\partial_t\phi^{ac}$ is the absolutely continuous part of the measure $\partial_t\phi$. 
\end{corollary} 

\begin{proof} Let $\partial_t\phi^{s}$ be the singular part of the measure $\partial_t\phi$. Since inequality $-\partial_t\phi\leq -H(x,Du)+ \alpha$ holds in the sense of distribution and the right-hand side of the inequality belongs to $L^1$, $\partial_t\phi^{s}\geq 0$. Assume that the result claimed in the corollary is false. Then there exists $\delta>0$ such that $-\partial_t\phi^{ac}(t,x)+H(x,D\phi(t,x))\leq \alpha(t,x)-\delta$ in a set $E$ in which $m\geq \delta$ a.e. and such that $|E|\geq \delta$. If we set $\alpha_\delta:= \alpha-\delta {\bf 1}_E$,  then, as  $\partial_t\phi^{s}\geq 0$, one has $-\partial_t\phi+H(x,Du)\leq \alpha_\delta$, so that the pair $(\phi,\alpha_\delta)$ still belongs to $\s{K}$. Lemma \ref{lem:integration_by_parts}  implies that 
$$
\int_{\bb{T}^d} \phi(0,x)m_0(x)dx \leq \int_{\bb{T}^d} \phi_T(x)m(T,x)dx + \int_0^T \int_{\bb{T}^d} m(\tau,x)L\left(x,\frac{w(\tau,x)}{m(\tau,x)}\right) + \alpha_\delta(\tau,x)m(\tau,x)dx d\tau, 
$$
where 
$$
\int_0^T \int_{\bb{T}^d} \alpha_\delta(\tau,x)m(\tau,x)dx d\tau\leq \int_0^T \int_{\bb{T}^d} \alpha(\tau,x)m(\tau,x)dx d\tau- \delta^3.
$$
This contradicts equality \eqref{lkqhsdn}.

\end{proof}

\subsection{Existence of minimizers of the relaxed problem}
\label{subsec:existence_relaxed}

We now explain that the relaxed problem has a solution: \medskip
\begin{theorem} \label{thm:existence} Under our structure conditions, 
Problem \ref{pr:relaxed} has a minimizer $(\phi,\alpha) \in \s{K}$.
\end{theorem} \medskip

Before actually proving the result, we need pointwise estimates on solutions of the Hamilton-Jacobi equation. \medskip

\begin{lemma} \label{lem:upper_bound}
For any smooth subsolution of $-\partial_t \phi + H(x,D\phi) \leq \alpha$, for $0 \leq t_1 < t_2 \leq T$ and $x,y \in \bb{T}^d$ we have
\begin{equation} \label{eq:upper_bound}
\phi(t_1,x) \leq \phi(t_2,y) + C\left[|x-y|^{r'}(t_2-t_1)^{1-r'} + \left[(t_2-t_1)^{\nu}\wedge 1+ T^{1/q}\right] (\|\alpha_+\|_p+1)\right]
\end{equation}
where $C$ does not depend on $T$ and 
\begin{equation} \label{eq:nu}
\nu := \frac{r-d(q-1)}{d(q-1)(r-1)+rq}.
\end{equation}
\end{lemma} \medskip

Note that, by assumption \eqref{eq:conqr}, $\nu$ is positive. 

\begin{proof}
Fix $\beta \in (1/r,\frac{1}{d(q-1)})$ to be specified later. Let $x,y \in \bb{T}^d$ and $0 \leq t_1 < t_2 \leq T$ such that $t_2-t_1\leq 1$ (this assumption is removed at the end of the proof), and let $\gamma$ be an absolutely continuous path such that 
\begin{itemize}
\item $\gamma(t_1) = x$,
\item $\gamma(t_2) = y$, and
\item $C_\gamma := \int_{t_1}^{t_2} |\dot{\gamma}(s)|^{r'} ds < \infty$.
\end{itemize} 
For instance, we can choose $\gamma(s) = x + \theta(s-t)$, where $\theta = \displaystyle \frac{x-y}{t_2-t_1}$; this example gives the minimial value of $C_\gamma$ at $|x-y|^{r'}(t_2-t_1)^{1-r'}$.

For any $\sigma \in \bb{R}^d$ with $|\sigma| \leq 1$ define the arc
\begin{equation}
x_\sigma(s) =  \left\{\begin{array}{ccc}
\gamma(s) + \sigma(s-t_1)^\beta & \text{if} & s \in [t_1,\frac{t_1+t_2}{2}]\\
\gamma(s) + \sigma(t_2-s)^\beta & \text{if} & s \in [\frac{t_1+t_2}{2},t_2]
\end{array} \right.
\end{equation}
Then
\begin{align*}
& \hspace{-1cm}\frac{d}{ds}\left[\phi(s,x_\sigma(s)) - \int_s^{t_2} L(x_\sigma(\tau),\dot{x}_\sigma(\tau))d\tau \right]\\
&= \partial_t \phi(s,x_\sigma(s)) + D\phi(s,x_\sigma(s)) \cdot \dot{x}_\sigma(s) + L(x_\sigma(s),\dot{x}_\sigma)\\
&\geq \partial_t \phi(s,x_\sigma(s)) - H(x_\sigma(s),D\phi(s,x_\sigma(s))) \geq - \alpha_+(s,x_\sigma(s)).
\end{align*}
Integrating over $[t_1,t_2] \times B_1$ we get
\begin{equation}
\phi(t_1,x) \leq \phi(t_2,y) + \frac{1}{|B_1|}\int_{B_1} \int_{t_1}^{t_2} [L(x_\sigma(s),\dot{x}_\sigma(s)) + \alpha_+(s,x_\sigma))]ds d\sigma.
\end{equation}
By Assumption (\ref{eq:hamiltonian_conjugate_bounds}) we have
\begin{equation*}
\frac{1}{|B_1|}\int_{B_1} \int_{t_1}^{t_2} L(x_\sigma(s),\dot{x}_\sigma(s))ds d\sigma \leq C\left[\int_{B_1} \int_{t_1}^{t_2} |\dot{x}_\sigma(s)|^{r'}ds d\sigma + (t_2-t_1)\right].
\end{equation*}
We compute
\begin{align*}
\int_{B_1} \int_{t_1}^{\frac{t_1 + t_2}{2}} |\dot{x}_\sigma(s)|^{r'}ds d\sigma &= \int_{B_1} \int_{t_1}^{\frac{t_1 + t_2}{2}} |\dot{\gamma}(s) + \beta^{r'}\sigma(s-t_1)^{\beta - 1}|^{r'}ds d\sigma\\
&\leq C(C_\gamma + \beta^{r'}(t_2-t_1)^{1-r'(1-\beta)})
\end{align*}
where $1-r'(1-\beta) > 0$. 
In like manner we obtain
\begin{equation}
\int_{B_1} \int_{\frac{t_1 + t_2}{2}}^{t_2} |\dot{x}_\sigma(s)|^{r'}ds d\sigma \leq C(C_\gamma + \beta^{r'}(t_2-t_1)^{1-r'(1-\beta)}),
\end{equation}
from which we deduce
\begin{equation}
\frac{1}{|B_1|}\int_{B_1} \int_{t_1}^{t_2} L(x_\sigma(s),\dot{x}_\sigma(s))ds d\sigma \leq C(C_\gamma + \beta^{r'}(t_2-t_1)^{1-r'(1-\beta)}).
\end{equation}
On the other hand, using H\"{o}lder, we get
\begin{align*}
\int_{B_1} \int_{t_1}^{\frac{t_1 + t_2}{2}} \alpha_+(s,x_\sigma(s))ds d\sigma &\leq \int_{t_1}^{\frac{t_1 + t_2}{2}}\int_{B_1} \alpha_+(s,x + \theta(\tau - t_1) + \sigma(s-t_1)^\beta) d\sigma ds\\
&\leq \int_{t_1}^{\frac{t_1 + t_2}{2}}\int_{B(x+\theta(\tau-t_1);(s-t_1)^\beta)} (s-t_1)^{-d\beta}\alpha_+(s,\rho) d\rho ds\\
&\leq \left[\int_{t_1}^{\frac{t_1 + t_2}{2}}C(s-t_1)^{-d\beta(q-1)}ds\right]^{\frac{1}{q}} \|\alpha_+\|_p\\
&\leq C(t_2-t_1)^{(1-d\beta(q-1))/q}\|\alpha_+\|_p.
\end{align*}
Applying the same argument on the time interval $[\frac{t_1+t_2}{2},t_2]$ and summing we can conclude
\begin{equation}
\frac{1}{|B_1|}\int_{B_1} \int_{t_1}^{t_2} \alpha(s,x_\sigma(s))ds d\sigma \leq C(t_2-t_1)^{(1-d\beta(q-1))/q}\|\alpha_+\|_p.
\end{equation}
We now specify that
\begin{equation}
\beta := \frac{q(r'-1)+1}{d(q-1)+r'q} = \frac{q+r-1}{d(q-1)(r-1)+rq}.
\end{equation}
Note that $\beta \in (1/r,\frac{1}{d(q-1)})$ and
\begin{equation}
\frac{1-d\beta(q-1)}{q} = 1-r'(1-\beta) = \frac{r-d(q-1)}{d(q-1)(r-1)+rq} = \nu.
\end{equation} It follows from the above argument that
\begin{equation}
\phi(t_1,x) \leq \phi(t_2,y) + C[C_\gamma + (t_2-t_1)^{\nu}(\|\alpha_+\|_p+1)].
\end{equation}
Taking the particular example of $\gamma(s) = x + \theta(s-t)$, where $\theta = \displaystyle \frac{x-y}{t_2-t_1}$, we conclude with Equation (\ref{eq:upper_bound}).

It remains to dispose of the assumption $t_2-t_1\leq 1$: if $t_2-t_1\geq 1$, we can argue in the same way, just changing the family of paths into 
$$
x_\sigma(s) =  \left\{\begin{array}{ccc}
\gamma(s) + \sigma(s-t_1)^\beta & \text{if} & s \in [t_1,t_1+1/2]\\
\gamma(s) + \sigma(1/2)^\beta & \text{if} & s \in [t_1+1/2, t_2-1/2]\\
\gamma(s) + \sigma(t_2-s)^\beta & \text{if} & s \in [t_2-1/2,t_2]
\end{array} \right.
$$
\end{proof}

\begin{proof}[Proof of Theorem \ref{thm:existence}.]
Let $\phi_n \in \s{K}_0$ be minimizing sequence, that is, suppose
\begin{equation}\label{eq:minseq}
\int_0^T \int_{\bb{T}^d} F^*(x,\alpha_n) dx dt - \int_{\bb{T}^d} \phi_n(0,x)m_0(x)dx \to \inf_{\phi \in \s{K}_0} \s{A}(\phi)
\end{equation}
where $\alpha_n := -\partial_t \phi_n + H(x,D\phi_n)$.

{\em Step 1.} We need to prove that $(\alpha_n)$ is bounded in an $L^1$ sense. First note that
\begin{equation} \label{eq:jnejlnvr}
C \geq \int_0^T \int_{\bb{T}^d} F^*(x,\alpha_n) dx dt - \int_{\bb{T}^d} \phi_n(0,x)m_0(x)dx.
\end{equation}
By the bounds on $F^*$ in Equation (\ref{eq:cost_growth_star}) we have
\begin{equation} \label{eq:alpha_bound0}
\int_0^T \int_{\bb{T}^d} F^*(x,\alpha_n) dx dt \geq \frac{1}{pC}\|(\alpha_n)_+\|_p^p + \iint_{\alpha_n < 0} F^*(x,-(\alpha_n)_-) dx dt -C,
\end{equation}
while by Lemma \ref{lem:upper_bound}, there is a constant $C > 0$ such that
\begin{equation} \label{eq:alpha_bound1}
- \int_{\bb{T}^d} \phi_n(0,x)m_0(x)dx \geq  - C\|(\alpha_n)_+\|_p  - C.
\end{equation}
Let $\epsilon > 0$ to be chosen later. Since $F(\cdot,\epsilon)$ is continuous on $\bb{T}^d$, we let $K_\epsilon$ be its supremum. We have $F^*(x,a) \geq \epsilon a - K_\epsilon$ by definition of Fenchel conjugate. So from Equation (\ref{eq:alpha_bound0}) we get
\begin{equation} \label{eq:alpha_bound2}
\int_0^T \int_{\bb{T}^d} F^*(x,\alpha_n) dx dt \geq \frac{1}{pC}\|(\alpha_n)_+\|_p^p - \epsilon \int_0^T \int_{\bb{T}^d} (\alpha_n)_- dx dt - K_\epsilon T - C.
\end{equation}
Now taking into account (\ref{eq:hamiltonian_bounds}) we have
\begin{equation} \label{eq:alpha_minus}
(\alpha_n)_- = (\alpha_n)_+ + \partial_t \phi_n - H(x,\phi_n) \leq (\alpha_n)_+ + \partial_t \phi_n + C,
\end{equation}
and upon integration we get
\begin{align*}
\frac{1}{C}\int_0^T \int_{\bb{T}^d} (\alpha_n)_- dx dt &\leq \int_0^T \int_{\bb{T}^d} (\alpha_n)_- m_0 dx dt\\
&\leq \int_0^T \int_{\bb{T}^d} (\alpha_n)_+ m_0 dx dt + \int_{\bb{T}^d} \phi_T m_0 - \phi_n(0)m_0 dx + CT\\
&\leq C\|(\alpha_n)_+\|_p + C(T+1) - \int_0^T \int_{\bb{T}^d} F^*(x,\alpha_n) dx dt,
\end{align*}
where we have used \eqref{eq:jnejlnvr} for the last inequality. 
Substituting into (\ref{eq:alpha_bound2}) we get
\begin{equation} \label{eq:alpha_bound3}
(1-\epsilon C)\int_0^T \int_{\bb{T}^d} F^*(x,\alpha_n) dx dt \geq \frac{1}{pC}\|(\alpha_n)_+\|_p^p - \epsilon C^2 \|(\alpha_n)_+\|_p - K_\epsilon T - \epsilon C^2(T+1).
\end{equation}
Take $\epsilon >0$ sufficiently small. Combining (\ref{eq:alpha_bound3}) and (\ref{eq:alpha_bound1}) and substituting into (\ref{eq:jnejlnvr}) we get
\begin{equation} \label{eq:alpha_bound4}
C \geq \frac{1}{pC}\|(\alpha_n)_+\|_p^p - C\|(\alpha_n)_+\|_p - C K_\epsilon T
\end{equation}
where $C$ is some large constant. It follows that $(\alpha_n)_+$ is uniformly bounded in $L^p$.

To show that $(\alpha_n)_-$ is uniformly bounded in $L^1$, we will use (\ref{eq:alpha_minus}). First, for some constant $C$ large enough and any $\epsilon > 0$ we have
\begin{align*}
C &\geq \int_0^T \int_{\bb{T}^d} F^*(x,\alpha_n) dx dt - \int_{\bb{T}^d} \phi_n(0,x)m_0(x)dx\\
&\geq -\epsilon \int_0^T \int_{\bb{T}^d} (\alpha_n)_- dx dt - K_\epsilon T  - \int_{\bb{T}^d}\phi_n(0,x)m_0(x)dx\\
&\geq -\epsilon C\|(\alpha_n)_+\|_p - \epsilon C(T+1) - K_\epsilon T  - (1-\epsilon C) \int_{\bb{T}^d}\phi_n(0,x)m_0(x)dx.
\end{align*}
It follows that $\dys - \int_{\bb{T}^d} \phi_n(0)m_0$ is uniformly bounded above. We make the stronger claim that $\|\phi_n(0)\|_{L^1(\bb{T}^d)}$ is uniformly bounded. By Lemma \ref{lem:upper_bound}, 
$$
\phi_n(t,x)\leq \phi_T(x)+ CT^\nu (\|(\alpha_n)_+\|_p+1)\leq C,
$$
so $\phi_n$ is bounded above. Then using $1/C\leq m_0\leq C$ we find
$$
 \int_{\bb{T}^d} |\phi_n(0)| \leq  \int_{\bb{T}^d} (\phi_n(0))_-+C \leq C  \int_{\bb{T}^d} (\phi_n(0))_-m_0 +C \leq -C\int_{\bb{T}^d} \phi_n(0)m_0+ C \leq C,
$$
as desired.
Now Equation (\ref{eq:alpha_minus}) implies
\begin{equation}
\int_0^T \int_{\bb{T}^d} (\alpha_n)_- \leq C + C\|(\alpha_n)_+\|_p + \int_{\bb{T}^d} (\phi_T - \phi_n(0))
\end{equation}
which proves that $(\alpha_n)_-$ is uniformly bounded in $L^1$. We conclude that $(\alpha_n)$ is bounded in $L^1$.

{\em Step 2.} 
We would like to show next that $\phi_n$ is uniformly bounded in $L^1$. We already have from (\ref{eq:upper_bound}) that $\phi_n$ has a uniform upper-bound. On the other hand, one has from the same equation that
\begin{align*}
\phi_n(t,x) &\geq  \phi_n(0,x) - Ct^\nu(\|(\alpha_n)_+\|_p + 1) \; \geq \; \phi_n(0,x) - C.
\end{align*}
As $\|\phi_n(0)\|_1$ is bounded, this implies that $\dys \sup_{t\in [0,T]}\|\phi_n(t)\|_1\leq C$.

{\em Step 3.}
We want to take a modified version of $\alpha_n$ which has a weak limit in $L^1$ but also still minimizes the target functional. Since $(\alpha_n)_-$ is bounded in $L^1$, by \cite{kosygina08} we can pass to a subsequence on which $(\alpha_n)_- = \beta_n + r_n,$ where 
\begin{itemize}
\item $\beta_n = (\alpha_n)_- {\bf 1}_{\{(\alpha_n)_- \leq l_n\}}$ where $l_n \to \infty$,
\item $(\beta_n)$ is uniformly integrable,
\item $\beta_n r_n = 0$,  $r_n \geq 0$ and the measure of $\{r_n > 0\}$ goes to zero.
\end{itemize} 
Observe that $[0,T] \times \bb{T}^d$ can be decomposed as $\{\alpha_n \geq 0\} \cup \{\beta_n > 0\} \cup \{r_n > 0\}$, where the union is disjoint and the set $\{r_n > 0\}$ coincides with $\{(\alpha_n)_- > l_n\}$. Define $\tilde{\alpha}_n := \alpha_n $ on $\{r_n=0\}$ and $\tilde{\alpha}_n := 1$ otherwise. We have
\begin{multline}
\int_0^T \int_{\bb{T}^d} F^*(x,\alpha_n) = \iint_{\{\alpha_n \geq 0\}} F^*(x,(\alpha_n)_+) + \iint_{\{\beta_n > 0\}} F^*(x,-\beta_n) + \iint_{\{r_n > 0\}} F^*(x,-r_n)\\ = \int_0^T \int_{\bb{T}^d} F^*(x,\tilde{\alpha}_n) + \iint_{\{r_n > 0\}} (F^*(x,-r_n)-F^*(x,1)).
\end{multline}
To see that $(\phi_n,\tilde{\alpha}_n)$ is also a minimizing sequence in $\s{K}$, it suffices to show that 
$$
\liminf_{n \to \infty} \iint_{\{r_n > 0\}} F^*(x,-r_n)-F^*(x,1) = 0.$$
For $\epsilon > 0$, let us set as usual $K_\ep=\max_{\bb{T}^d}F(\cdot,\ep)$. Then
\begin{align*}
\iint_{\{r_n > 0\}} F^*(x,-r_n) -F^*(x,1) &\geq \iint_{\{(\alpha_n)_- > l_n\}} (-\epsilon r_n - K_\epsilon-\max_{\bb{T}^d}F^*(\cdot,1))\\
&\geq - \epsilon \|\alpha_n\|_{L^1} - (K_\epsilon+C) \frac{1}{l_n} \iint_{\{(\alpha_n)_- > l_n\}} (\alpha_n)_-\\
&\geq -(\epsilon + \frac{K_\epsilon+C}{l_n})\|\alpha_n\|_{L^1}
\end{align*}
Since $(\alpha_n)$ is bounded in $L^1$, $l_n \to \infty$, and $\epsilon > 0$ is arbitrary, the claim follows. We have shown that $(\phi_n,\tilde{\alpha}_n)$ is a minimizing sequence of $\s{A}$ in $\s{K}$ with $(\tilde{\alpha}_n)$ uniformly integrable.

{\em Step 4.} We want to show that $(\phi_n)$ is bounded in $BV$, and that its weak$^*$ limit has the desired regularity. First of all,
\begin{align*}
\int_0^T \int_{\bb{T}^N} H(x,D\phi_n(t,x)) dx dt &= \int_0^T\int_{\bb{T}^N} \partial_t \phi_n(t,x) + \alpha_n(t,x) dxdt \\
&\leq \int_{\bb{T}^N}\phi_T(x) - \phi_n(0,x)dx + \|\alpha_n\|_1
\end{align*}
which proves by (\ref{eq:hamiltonian_bounds}) that $D\phi_n$ is uniformly bounded in $L^{r'}$ and $H(x,D\phi_n)$ is uniformly bounded in $L^1$.
We then have that $\partial_t \phi_n = H(x,D\phi_n) - \alpha_n$ is bounded in $L^1$ as well. Thus $\phi_n$ is bounded in $BV$ and we have some $\phi \in BV$ for which $\phi_n \to \phi$ in $L^1$, and $(\partial_t \phi_n,D\phi_n) \rightharpoonup (\partial_t \phi,D\phi)$ weakly in the sense of measures. Finally, we prove that $\phi \in L^\infty(t,T)\times \T^d)$ for almost every $t \in [0,T]$. Indeed, we already proved uniform upper bounds on $\phi_n$, which thus apply to $\phi$. On the other hand,
\begin{align*}
-\phi_n(t,x) &= - \int_{\bb{T}^d} \phi_n(t,x) m_0(y)dy\\
&\leq - \int_{\bb{T}^d} \phi_n(0,y) m_0(y)dy + Ct^{1-r'}\int_{\bb{T}^d} |y-x|^{r'} m_0(y)dy + Ct^\nu(\|(\alpha_n)_+\|_p + 1)
\end{align*}
which shows that $\phi_n$ is uniformly bounded below on $[t,T] \times \bb{T}^d$ for all $t > 0$.

{\em Step 5.} Finally, we pass to the limit and verify the existence of a minimizer. Since $\tilde{\alpha}_n$ is uniformly integrable, by the Dunford-Pettis theorem there exists $\alpha \in L^1$ such that (up to a subsequence) $\tilde{\alpha}_n \rightharpoonup \alpha$ weakly in $L^1$. Since
\begin{equation}
-\partial_t \phi_n + H(x,D\phi_n) \leq \tilde{\alpha}_n
\end{equation}
for each $n$, we pass to the limit and get
\begin{equation}
-\partial_t \phi + H(x,D\phi) \leq \alpha
\end{equation}
in the sense of distributions. Finally the terminal condition holds, thanks to Lemma \ref{lem:upper_bound}. 

It remains to check that $(\phi, \alpha)$ minimizes \eqref{pr:relaxed}. We just proved that $(\phi, \alpha)\in {\mathcal K}$. By convexity of $F^*$ and weak convergence of $(\tilde{\alpha}_n)$, 
$$
\int_0^T \int_{\bb{T}^d} F^*(x,\alpha) dx dt \leq \liminf_n \int_0^T \int_{\bb{T}^d} F^*(x,\tilde{\alpha}_n) dx dt.
$$
Fix $\ep>0$ small. By (\ref{eq:upper_bound}),  
$$
-\int_{\bb{T}^d} \phi_n(0)m_0\geq - \frac{1}{\ep}\int_0^\ep \int_{\bb{T}^d} \phi_n(s,x)m_0(x)dxds -C \ep^{\nu}.
$$
So, by $L^1$ convergence of $(\phi_n)$,  
$$
\liminf_n -\int_{\bb{T}^d} \phi_n(0)m_0\geq - \frac{1}{\ep}\int_0^\ep \int_{\bb{T}^d} \phi(x,s)m_0(x)dxds -C \ep^{\nu}.
$$
Recalling \eqref{eq:minseq} we finally have 
$$
\int_0^T \int_{\bb{T}^d} F^*(x,\alpha) dx dt- \frac{1}{\ep}\int_0^\ep \int_{\bb{T}^d} \phi(x,s)m_0(x)dxds -C \ep^{\nu}
\leq \inf_{\phi \in \s{K}_0} \s{A}(\phi), 
$$
which, after letting $\ep\to0$, implies the desired result. 

\end{proof}

A straightforward consequence of the proof is the following approximation of the optimal pair $(\phi,\alpha)$ of Problem \ref{pr:relaxed}. \medskip

\begin{corollary}\label{cor:approx} There exists  $(\phi,\alpha)\in \s{K}$ which is optimal for Problem \ref{pr:relaxed} and for which there exists a sequence of  maps $(\phi^n, \alpha^n)\in \s{K}$, such that, for each $n$, $\phi^n$ is  $C^1$,  $(\phi^n)$  is bounded in $BV$ and in $L^\infty([t,T]\times \bb{T}^d)$ for any $t\in (0,T)$, $(\phi^n)$ converges to $\phi$ in $L^1$ while $(D\phi^n)$ converge to $(D\phi)$ weakly in $L^r$,  $(\alpha^n_+)$ is bounded in $L^p$,  $(\alpha^n)$ converges weakly in $L^1$ to $\alpha$ while 
$$
\int_0^T\int_{\bb{T}^d} F^*(x,\alpha^n)dxdt \to \int_0^T\int_{\bb{T}^d} F^*(x,\alpha)dxdt.
$$
\end{corollary}

\section{Weak solutions} \label{sec:weak_solutions}

The purpose of this section is to define and characterize the solutions of System (\ref{eq:mfg}).
It is divided into three subsections.
We define weak solutions in Section \ref{sec:definition_weak}. Then we prove their existence and (partial) uniqueness in Section \ref{sec:existence_weak}.
Finally, we show that the solution  satisfies a classical energy estimate (Section \ref{sec:energy_inequality}).
The energy estimate will be applied later in Section \ref{sec:asymp} to study long time average behavior.

\subsection{Definition of weak solutions}
\label{sec:definition_weak}

\begin{definition} \label{def:weak}
A pair $(\phi,m) \in BV((0,T) \times \bb{T}^d) \times L^q((0,T) \times \bb{T}^d)$ is called a weak solution to the system (\ref{eq:mfg}) if it satisfies the following conditions.
\begin{enumerate}
\item $D\phi\in L^r$ and  the maps $mf(x, m)$, $mH^*\left(x,-D_pH(x,D\phi)\right) $ and   $mD_p H(x,D\phi)$ are integrable, 
\item $\phi$ satisfies  a first-order Hamilton-Jacobi inequality
\begin{equation} \label{eq:hjb_weak}
-\partial_t \phi + H(x,D\phi) \leq  f(x,m) 
\end{equation}
in the sense of distributions, the boundary condition $\phi(T,\cdot)\leq \phi_T$ in the sense of trace and the following equality
\begin{multline}
\int_0^T \int_{\bb{T}^d}m(t,x) \left( H(x,D\phi(t,x))-\lg D\phi(t,x), D_pH(x, D\phi(t,x))\rg -f(x,m(t,x))\right)dxdt \\
= \int_{\bb{T}^d} (\phi_T(x)m(T,x))-\phi(0,x)m_0(x))dx \label{eq:ibp_weak}
\end{multline}

\item $m$ satisfies the continuity equation
\begin{equation} \label{eq:continuity_weak}
\partial_t m - \mathrm{div}~(mD_p H(x,D\phi)) = 0 ~~~\text{in}~~(0,T) \times \bb{T}^d, ~~~ m(0) = m_0
\end{equation}
in the sense of distributions.
\end{enumerate}
\end{definition} \medskip

\begin{remark}{\rm 
From Corollary \ref{cor:egalite}, we have by \eqref{eq:ibp_weak} that 
\begin{equation} \label{eq:almost_everywhere}
-\partial_t\phi^{ac}(t,x)+H(x,D\phi(t,x))= f(x,m(t,x)) ~~ m-\text{a.e. in}~~(0,T)\times \bb{T}^d,
\end{equation}
where $\partial_t\phi^{ac}$ is the absolutely continuous part of the measure $\partial_t\phi$. 
}\end{remark} \medskip

\begin{remark}\label{remdef}{\rm  Again in view of \eqref{eq:ibp_weak} and Lemma \ref{lem:integration_by_parts}, inequalities \eqref{eq:integration_by_parts} and \eqref{eq:integration_by_parts2} are actually equalities for the solution $(\phi, \alpha:= f(\cdot, m))$. 
}\end{remark}

Let us make a few comments on the definition of weak solution, which draws its inspiration from \cite{cardaliaguet2012geodesics} and is in the same vein as \cite{cardaliaguet2013weak,graber2013optimal}.
First, we mention that the integrals in (\ref{eq:ibp_weak}) are well-defined: because $mf(\cdot,m)$ and $mH^*(\cdot,-D_pH(x,D\phi))$ are integrable, so the fact that $m \geq 0$ and the definition of Fenchel conjugate prove that the left-hand side is integrable, while the right-hand side is integrable by assumptions on $\phi_T$ and $m_0$ (namely that they are bounded).
We note in particular that $f(\cdot,m)$ is in $L^p$ be the growth assumptions on $f$ (\ref{eq:cost_growth}) and that $H(\cdot,D\phi)$ is integrable by the growth assumptions on the Hamiltonian (\ref{eq:hamiltonian_bounds}) and the fact that $D\phi \in L^r$.
This gives meaning to Equation (\ref{eq:hjb_weak}).

Next, we comment on the above definition in relation to System (\ref{eq:mfg}).
The continuity equation is naturally dealt with in part 3 of Definition \ref{def:weak}.
On the other hand, the Hamilton-Jacobi equation is more difficult to interpret, as pointed out in the references \cite{cardaliaguet2012geodesics,cardaliaguet2013weak,graber2013optimal}.
In particular, we cannot expect the Hamilton-Jacobi equation to hold in the viscosity sense, since neither $\phi$ nor $m$ is continuous.
Equation (\ref{eq:almost_everywhere}) is close to giving an ``almost everywhere" sense to the Hamilton-Jacobi equation, but only on $\{m > 0\}$ and for the absolutely continuous part of $\partial_t \phi$.

In practice it is often useful to know that a solution can be approximated by smoother maps. This is the aim of the next definition: 

\begin{definition}\label{def:good} We say that a weak solution $(\phi, m)$ to the MFG system is ``good" if there exists a sequence $(\phi^n, \alpha^n)$ such that 
for each $n$, $\phi^n$ is  $C^1$,  $(\phi^n)$  is bounded in $BV$ and in $L^\infty([t,T]\times \bb{T}^d)$ for any $t\in (0,T)$, $(\phi^n)$ converges to $\phi$ in $L^1$ while $(D\phi^n)$ converge to $(D\phi)$ weakly in $L^r$, $(\alpha^n_+)$ is bounded in $L^p$, $(\alpha^n)$ converges weakly in $L^1$ to $\alpha$ while 
$$
\int_0^T\int_{\bb{T}^d} F^*(x,\alpha^n)dxdt \to \int_0^T\int_{\bb{T}^d} F^*(x,\alpha)dxdt.
$$
\end{definition}

\subsection{Existence and uniqueness of weak solutions}
\label{sec:existence_weak}

The main result of this section is the following. \medskip
\begin{theorem}[Existence and (partial) uniqueness of weak solutions] \label{thm:minimizers_weak}
(i) If $(m,w) \in \s{K}_1$ is a minimizer of Problem \ref{eq:dual} and $(\phi,\alpha) \in \tilde{\s{K}}$ is a minimizer of  Problem  \ref{pr:relaxed}, then $(\phi,m)$ is a weak solution of (\ref{eq:mfg}) and $\alpha(t,x) = f(x,m(t,x))$ almost everywhere. 


(ii) Conversely, if $(\phi,m)$ is a weak solution of (\ref{eq:mfg}), then there exist functions $w, \alpha$ such that $(\phi,\alpha) \in \s{K}$ is a minimizer of Problem  \ref{pr:relaxed} and $(m,w) \in \s{K}_1$ is a minimizer of Problem  \ref{pr:dual}.

(iii) If $(\phi,m)$ and $(\phi',m')$ are both weak solutions to (\ref{eq:mfg}), then $m = m'$ almost everywhere while $\phi = \phi'$ almost everywhere in the set $\{m > 0\}$.
\end{theorem}\medskip

\begin{remark}{\rm The existence of a minimizer $(\phi,\alpha) \in \s{K}$ of Problem \ref{pr:relaxed} such that $\phi \in L^\infty([t,T]\times \bb{T}^d)$ for any $t \in (0,T)$ is guaranteed by Theorem \ref{thm:minimizers_weak}. Moreover, Corollary \ref{cor:approx} states that there exists ``good" solutions. 
}\end{remark} \medskip

\begin{remark}{\rm If we assume that $f$ is bounded below, then following \cite{cardaliaguet2013weak} one can show the existence of a solution $(\phi,m)$ such that $\phi$ is  continuous  and satisfies the terminal condition $\phi(T,\cdot)=\phi_T$. 
}\end{remark}

\begin{proof}
(i) {\em Step 1.} Suppose $(m,w) \in \s{K}_1$ is a minimizer of Problem  \ref{eq:dual} and $(\phi,\alpha) \in \s{K}$ is a minimizer of Problem  \ref{pr:relaxed}.
By Proposition \ref{prop:relaxed} we have
\begin{equation}
\int_0^T \int_{\bb{T}^N} F^*(x,\alpha) + F(x,m) + mL\left(x,\frac{w}{m}\right) dxdt + \int_{\bb{T}^N} \phi_T m(T) - \phi(0)m_0 dx = 0.
\end{equation}
We claim that $\alpha = f(x,m)$. In general
\begin{equation} \label{eq:F_Fstar_geq}
F^*(x,\alpha(t,x)) + F(x,m(t,x)) - \alpha(t,x)m(t,x) \geq 0, 
\end{equation}
so that 
$$
\int_0^T \int_{\bb{T}^N}  \alpha(t,x)m(t,x)+ mL\left(x,\frac{w}{m}\right) dxdt + \int_{\bb{T}^N} \phi_T m(T) - \phi(0)m_0 dx \leq 0.
$$
Using Lemma \ref{lem:integration_by_parts}, this inequality has to be an equality and therefore  the inequality in Equation (\ref{eq:F_Fstar_geq}) is in fact equality almost everywhere: hence, we have both
\begin{equation} \label{eq:alpha_equals_f}
\alpha(t,x) = f(x,m(t,x))
\end{equation}
almost everywhere and
\begin{equation} \label{eq:ibp_equality}
\int_0^T \int_{\bb{T}^N} \alpha m + mL\left(x,\frac{w}{m}\right) dxdt + \int_{\bb{T}^N} \phi_T m(T) - \phi(0)m_0 dx = 0.
\end{equation}
Applying Equations (\ref{eq:alpha_equals_f}) and (\ref{eq:ibp_equality}) to Equation (\ref{eq:integration_by_parts}) yields (\ref{eq:ibp_weak}). Moreover, considering that $(\phi,\alpha) \in \s{K}$ and Equation (\ref{eq:alpha_equals_f}), we have $-\partial_t \phi + H(x,D\phi) \leq f(x,m)$ in distribution and $\phi(T) \leq \phi_T$ in the sense of trace.

{\em Step 2.} We wish to show that (\ref{eq:continuity_weak}) holds. We know that
\begin{equation}
\partial_t m + \mathrm{div}~w = 0
\end{equation}
in distribution. It is enough to show therefore that
\begin{equation}
w = -mD_p H(x,D\phi),
\end{equation}
or equivalently that
\begin{equation}
- \langle w(t,x),D\phi(t,x) \rangle = m(t,x)H(x,D\phi(t,x)) + m(t,x)L\left(x,\frac{w(t,x)}{m(t,x)}\right)
\end{equation}
almost everywhere.

Let $m_\delta = m_{\epsilon(\delta),\delta}$ be the smooth approximation of $m$ obtained in the proof of Lemma \ref{lem:integration_by_parts} by convolution, and define $w_\delta$ to be the analogous smooth approximation of $w$. Since $-\partial_t \phi + H(x,D\phi) \leq \alpha$ in distribution, we have 
\begin{align*}
0 &\leq \iint mH(x,D\phi) + mL\left(x,\frac{w}{m}\right) + \langle w, D\phi \rangle \\
&\leq \liminf_{\delta \to 0} \iint m_\delta H(x,D\phi) + m L\left(x,\frac{w}{m}\right) + \langle w_\delta, D\phi \rangle\\
&\leq \liminf_{\delta \to 0} \iint (\partial_t \phi + \alpha)  m_\delta + m L\left(x,\frac{w}{m}\right) - (\mathrm{div}~ w_\delta) \phi
\end{align*}
using Fatou's Lemma in the second line. Since $m_\delta,w_\delta$ are obtained through convolution we have
\begin{equation}
\partial_t m_\delta + \mathrm{div}~w_\delta = 0,
\end{equation}
hence
\begin{align*}
0 &\leq \iint mH(x,D\phi) + mL\left(x,\frac{w}{m}\right) + \langle w, D\phi \rangle \\
&\leq \liminf_{\delta \to 0} \iint \partial_t \phi m_\delta + \phi \partial_t m_\delta + \alpha m_\delta + m L\left(x,\frac{w}{m}\right)\\
&= \liminf_{\delta \to 0} \int_{\bb{T}^d} \phi(T)m_\delta(T) - \phi(0)m_\delta(0) +  \iint \alpha m_\delta + m L\left(x,\frac{w}{m}\right)
\end{align*}
The proof of Lemma \ref{lem:integration_by_parts} shows that
\begin{equation}
\lim_{\delta \to 0} \int_{\bb{T}^d} \phi(T)m_\delta(T) - \phi(0)m_\delta(0) \leq \int_{\bb{T}^d} \phi_T m(T) - \phi(0)m_0.
\end{equation}
On the other hand, $m_\delta \to m$ in $L^q$, but $\alpha$ is not necessarily $L^p$. We can use $\alpha_M:= \alpha\vee (-M)$ (where $M$ is large) as an upper bound. Since $\alpha_M\in L^p$, we can let $\delta\to 0$ to get:
\begin{equation}
0 \leq \iint mH(x,D\phi) + mL\left(x,\frac{w}{m}\right) + \langle w, D\phi \rangle \leq \int_{\bb{T}^d} \phi_T m(T) - \phi(0)m_0 + \iint \alpha_M m + m L\left(x,\frac{w}{m}\right).
\end{equation}
The next step is to let $M \to \infty$. Note that $\alpha_M m \to \alpha m$ and $|\alpha_M m| \leq |\alpha m| = |\alpha| m$. Recall that $\alpha m$ is integrable from the proof of Lemma \ref{lem:integration_by_parts}.  
Hence by the dominated convergence theorem $\iint \alpha_M m \to \iint \alpha m$ and we have
\begin{equation}
0 \leq \iint mH(x,D\phi) + mL\left(x,\frac{w}{m}\right) + \langle w, D\phi \rangle \leq \int_{\bb{T}^d} \phi_T m(T) - \phi(0)m_0 + \iint \alpha m + m L\left(x,\frac{w}{m}\right) = 0
\end{equation}
by Equation (\ref{eq:ibp_equality}). This completes Part (i) of Theorem \ref{thm:minimizers_weak}.

(ii) Suppose now that $(\phi,m)$ is a weak solution of (\ref{eq:mfg}). Set $w = -mD_p H(x,D\phi)$ and $\alpha(t,x) = f(x,m(t,x))$. By definition of weak solution $\alpha \in L^1$.  Moreover, since $m \in L^q$ and $f$ is increasing in $m$, it follows by growth condition (\ref{eq:cost_growth}) that $\alpha_+ \in L^p$. We thus have that $(\phi,\alpha) \in \s{K}$, and $(m,w) \in \s{K}_1$. We want to show that $(\phi,\alpha)$ minimizes $\s{A}$ and $(m,w)$ minimizes $\s{B}$.

Begin with $(\phi,\alpha)$. Let $(\phi',\alpha') \in \s{K}$. By the convexity of $F$ in the second variable, we have
\begin{align*}
\s{A}(\phi',\alpha') &= \int_0^T \int_{\bb{T}^d} F^*(x,\alpha'(t,x))dx dt - \int_{\bb{T}^d} \phi'(0,x)m_0(x)dx\\
&\geq \int_0^T \int_{\bb{T}^d} F^*(x,\alpha(t,x)) + \partial_\alpha F^*(x,\alpha(t,x))(\alpha'(t,x) - \alpha(t,x)) dx dt - \int_{\bb{T}^d} \phi'(0,x)m_0(x)dx\\
&\geq \int_0^T \int_{\bb{T}^d} F^*(x,\alpha(t,x)) + m(t,x)(\alpha'(t,x) - \alpha(t,x)) dx dt - \int_{\bb{T}^d} \phi'(0,x)m_0(x)dx,
\end{align*}
where $\alpha' m$ and $\alpha m$ both belong to $L^1$ thanks to Lemma \ref{lem:integration_by_parts}. 
Using Equation (\ref{eq:ibp_weak}) we have
\begin{equation} \label{eq:ibp_2}
 \iint m \alpha =  - \iint mL\left(x,\frac{w}{m}\right) + \int_{\bb{T}^d} \phi(0)m_0 - \phi_T m(T).
\end{equation}
On the other hand, from Lemma \ref{lem:integration_by_parts} we have
\begin{equation}
\iint m\alpha' \geq - \iint mL\left(x,\frac{w}{m}\right) + \int_{\bb{T}^d} \phi'(0)m_0 - \phi_T m(T).
\end{equation}
Substituting into the previous estimate, we get
$$
\s{A}(\phi',\alpha') \geq \int_0^T \int_{\bb{T}^d} F^*(x,\alpha(t,x)) dx dt - \int_{\bb{T}^d} \phi(0,x)m_0(x)dx= \s{A}(\phi,\alpha),
$$
and $(\phi,\alpha)$ is a minimizer of $\s{A}$. 

The argument for $(m,w)$ is similar. Let $(m',w')$ minimize $\s{B}$. Then because $F$ is convex in the second variable, we have
\begin{align*}
\s{B}(m',w') &= \int_{\bb{T}^d} \phi_T m'(T) + \iint m'L\left(x,\frac{w'}{m'}\right) + F(x,m')\\
&\geq \int_{\bb{T}^d} \phi_T m'(T) + \iint m'L\left(x,\frac{w'}{m'}\right) + F(x,m) + f(x,m)(m'-m)\\
&= \int_{\bb{T}^d} \phi_T m'(T) + \iint m'L\left(x,\frac{w'}{m'}\right) + F(x,m) + \alpha(m'-m).
\end{align*}
Now we have just shown that $(\phi,\alpha)$ minimizes $\s{A}$. So by part (i), we have that $(\phi,m')$ is a weak solution of (\ref{eq:mfg}), and in particular
\begin{equation}
\iint \alpha m' + m'L\left(x,\frac{w'}{m'}\right) + \int_{\bb{T}^d} \phi_T m'(T) = \int_{\bb{T}^d} \phi(0)m_0.
\end{equation}
Combine this with (\ref{eq:ibp_2}) to get
\begin{equation}
\s{B}(m',w') \geq \int_{\bb{T}^d} \phi_T m(T) + \iint mL\left(x,\frac{w}{m}\right) + F(x,m) = \s{B}(m,w),
\end{equation}
as desired. This completes the proof of part (ii) of Theorem \ref{thm:minimizers_weak}.

(iii) The proof of uniqueness follows along the same lines Theorem 3.2. (iii) in \cite{graber2013optimal}, so we omit it.
\end{proof}

\subsection{An energy inequality} \label{sec:energy_inequality}

 The following energy inequality is standard for  solutions of MFG systems and goes back to the work of Lasry and Lions \cite{lasry07}. \medskip

\begin{proposition}\label{prop:energie}  Let $(\phi^1,m^1)$ and $(\phi^2,m^2)$ be two weak solutions of the MFG system with respective boundary conditions $(m^1_0,\phi^1_f)$ and $(m^2_0,\phi^2_f)$. We assume that  $(\phi^1,m^1)$ and $(\phi^2,m^2)$ are ``good" solutions in the sense of Definition \ref{def:good}. Then, for almost all $0\leq t_1\leq t_2\leq T$,  
$$
\begin{array}{l}
\ds \int_{t_1}^{t_2}\int_{\bb{T}^d} m^2\left( H(x,D \phi^1)-H(x,D \phi^2)-\lg D_pH(x,D\phi^2), D( \phi^1-\phi^2)\rg\right)\ dxdt \\
\ds \qquad +
\int_{t_1}^{t_2}\int_{\bb{T}^d} m^1\left( H(x, D\phi^2)-H(x, D \phi^1)-\lg D_pH(x,D  \phi^1), D(\phi^2-\phi^1)\rg\right)\ dxdt \\
\ds \qquad +
\int_{t_1}^{t_2}\int_{\bb{T}^d} (f(x,m^2)-f(x, m^1))(m^2-m^1) \ dxdt  \\
\qquad \qquad \ds \leq 
\left[\int_{\bb{T}^d} (m^2(t)- m^1(t)) (\phi^2(t)-\phi^1(t))\ dx\right]_{t_1}^{t_2}
\end{array}$$
\end{proposition}

Note that, as the right-hand side is well-defined, the left-hand side, which is nonnegative, converges: this is a part of the proposition. 
For smooth solutions, equality holds in the above inequality; we do not know if this is still the case for weak solutions. 

\begin{proof} Since   $(\phi^1,m^1)$ and $(\phi^2,m^2)$ are ``good" solutions, for $i=1,2$ we can approximate $(\phi^i,\alpha^i:= f(\cdot,m^i))$ by a sequence of maps  $(\phi^{i,n},\alpha^{i,n})$ such that $\phi^{i,n}$ is $C^1$ for each $n$, 
$(\phi^{i,n})$ is bounded in $BV$ and in $L^\infty([t,T]\times \bb{T}^d)$ for any $t\in (0,T)$, $(\phi^{i,n})$ converges to $\phi^i$ in $L^1$ while $(D\phi^{i,n})$ converge to $(D\phi^i)$ weakly in $L^r$, $(\alpha^{i,n}_+)$ is bounded in $L^p$, $(\alpha^{i,n})$ converges weakly in $L^1$ to $\alpha^i$ while 
$$
\int_0^T\int_{\bb{T}^d} F^*(x,\alpha^{i,n})dxdt \to \int_0^T\int_{\bb{T}^d} F^*(x,\alpha^i)dxdt
$$
and the following inequality holds a.e.:
\be\label{jhblhjhok}
-\partial_t \phi^{i,n}+H(x,D\phi^{i,n})\leq \alpha^{i,n}.
\ee

We multiply the equality $\partial_t m^1- {\rm div}(m^1 D_pH(x, D \phi^1))=0$ by $\phi^{2,n}$ and integrate on $(t_1,t_2)\times \T^d$:  
$$
\int_{t_1}^{t_2} \int_{\T^d} m^1 \partial_t \phi^{2,n}+\lg D\phi^{2,n}, m^1 D_pH(x, D \phi^1)\rg = \left[ \int_{\T^d} m^1(t)\phi^{2,n}(t)\right]_{t_1}^{t_2}. 
$$
By \eqref{jhblhjhok} for $i=2$: 
$$
\int_{t_1}^{t_2} \int_{\T^d} m^1 \partial_t \phi^{2,n}\geq \int_{t_1}^{t_2} \int_{\T^d} m^1( H(x,D\phi^{2,n})-\alpha^{2,n}), 
$$
(where, as $m^1\in L^q$ while $\alpha^{2,n}_+$ is in $L^p$, the left-hand side is well defined in $[-\infty,+\infty)$), so that 
\be\label{hbkljkzefhjA}
\int_{t_1}^{t_2} \int_{\T^d} m^1( H(x,D\phi^{2,n})-\alpha^{2,n}+\lg D\phi^{2,n}, D_pH(x, D \phi^1)\rg) \leq \left[ \int_{\T^d} m^1(t)\phi^{2,n}(t)\right]_{t_1}^{t_2}. 
\ee
Recalling Remark \ref{remdef}, we have, by \eqref{eq:ibp_weak} in the definition of weak solutions and the fact that inequalities \eqref{eq:integration_by_parts2} and \eqref{eq:integration_by_parts3} hold,  
\be \label{hbkljkzefhjB}
\int_{t_1}^{t_2} \int_{\bb{T}^d}m^1 \left( H(x,D\phi^1)-\lg D\phi^1, D_pH(x, D\phi^1)\rg -f(x,m^1)\right)
= \left[ \int_{\bb{T}^d} \phi^1 m^1 \right]_{t_1}^{t_2},
\ee
 for almost all $t_1<t_2$.
Combining \eqref{hbkljkzefhjA} with \eqref{hbkljkzefhjB} gives
\begin{multline}\label{hgbhjnnk}
\int_{t_1}^{t_2} \int_{\T^d}  m^1\left(H(x,D\phi^{2,n}) - H(x,D\phi^1)- \lg D_pH(x,D\phi^1) , D(\phi^{2,n}-\phi^1)\rg\right)\\
 +\int_{t_1}^{t_2} \int_{\T^d}  m^1\left( f(x,m^1)-\alpha^{2,n}\right)\leq   \left[\int_{\T^d} m^1(t)(\phi^{2,n}(t)-\phi^1(t))\right]_{t_1}^{t_2}.
\end{multline}
Our aim is now to let $n\to +\infty$ in \eqref{hgbhjnnk}: for the first integral, we use the fact that $D\phi^{2,n}$ weakly converges to $D\phi^2$ and that the map $\xi\to m^1\left(H(x,\xi) - H(x,D\phi^1)- \lg D_pH(x,D\phi^1) , \xi-D\phi^1)\right)$ is convex and nonnegative for any $x$ to get the inequality: 
\begin{multline*}
\int_{t_1}^{t_2} \int_{\T^d}  m^1\left(H(x,D\phi^{2}) - H(x,D\phi^1)- \lg D_pH(x,D\phi^1) , D(\phi^{2}-\phi^1)\rg\right) \\
\leq \liminf_n \int_{t_1}^{t_2} \int_{\T^d}  m^1\left(H(x,D\phi^{2,n}) - H(x,D\phi^1)- \lg D_pH(x,D\phi^1) , D(\phi^{2,n}-\phi^1)\rg\right)
\end{multline*}
For the second integral, we rewrite the last term as 
\be\label{jhebrehjj}
\int_{t_1}^{t_2} \int_{\T^d}  m^1\left(-\alpha^{2,n}\right) = 
\int_{t_1}^{t_2} \int_{\T^d}  \left(F^*(x,\alpha^{2,n})+F(x,m^1)-m^1\alpha^{2,n}\right) - \int_{t_1}^{t_2} \int_{\T^d} ( F^*(x,\alpha^{2,n})+F(x,m^1)).
\ee
The integrand in the first integral of the right-hand side is nonnegative and convex with respect to $\alpha^{2,n}$. As $(\alpha^{2,n})$ converges weakly in $L^1$ to $f(\cdot, m^2)$, we obtain the inequality 
$$
\int_{t_1}^{t_2} \int_{\T^d}  \left(F^*(x,f(x, m^2))+F(x,m^1)-m^1f(x, m^2)\right) 
\leq 
\liminf_n \int_{t_1}^{t_2} \int_{\T^d}  \left(F^*(x,\alpha^{2,n})+F(x,m^1)-m^1\alpha^{2,n}\right) 
$$
On another hand, we know by construction of $\alpha^{2,n}$ that the last integral in \eqref{jhebrehjj} converge as $n\to 0$ to $\ds \int_{t_1}^{t_2} \int_{\T^d} ( F^*(x,f(x, m^2))+F(x,m^1))$. Thus 
$$
\int_{t_1}^{t_2} \int_{\T^d}  m^1\left(-f(x, m^2)\right)\leq \liminf_n \int_{t_1}^{t_2} \int_{\T^d}  m^1\left(-\alpha^{2,n}\right). 
$$
We now consider the right-hand side of \eqref{hgbhjnnk}. Since the  $(\phi^{2,n})$ are uniformly bounded in $[t,T]$ for $t\in (0,T)$ and $(\phi^{2,n})$ converges a.e. to $\phi^{2}$, we have, for a.e. $t_1<t_2$,
$$
\lim_n  \left[\int_{\T^d} m^1(t)(\phi^{2,n}(t)-\phi^1(t))\right]_{t_1}^{t_2}
=  \left[\int_{\T^d} m^1(t)(\phi^{2}(t)-\phi^1(t))\right]_{t_1}^{t_2}.
$$
So, we can let $n\to \infty$ in \eqref{hgbhjnnk} to get, for a.e. $t_1<t_2$,
\begin{multline*}
\int_{t_1}^{t_2} \int_{\T^d}  m^1\left(H(x,D\phi^{2}) - H(x,D\phi^1)- \lg D_pH(x,D\phi^1) , D(\phi^{2}-\phi^1)\rg\right)\\
 +\int_{t_1}^{t_2} \int_{\T^d}  m^1\left( f(x,m^1)-f(x,m^2))\right)\leq   \left[\int_{\T^d} m^1(t)(\phi^{2}(t)-\phi^1(t))\right]_{t_1}^{t_2}.
\end{multline*}
Exchanging the roles of $(\phi^1,m^1)$ and $(\phi^2,m^2)$, we get 
\begin{multline*}
\int_{t_1}^{t_2} \int_{\T^d}  m^2\left(H(x,D\phi^{1}) - H(x,D\phi^2)- \lg D_pH(x,D\phi^2) , D(\phi^{1}-\phi^2)\rg\right)\\
 +\int_{t_1}^{t_2} \int_{\T^d}  m^2\left( f(x,m^2)-f(x,m^1))\right)\leq   \left[\int_{\T^d} m^2(t)(\phi^{1}(t)-\phi^2(t))\right]_{t_1}^{t_2}, 
\end{multline*}
which, added to the previous inequality gives the result. 
\end{proof}


\section{Analysis of the ergodic problem}\label{sec:2opti}

 The aim of this section is to define a notion of weak solution for the ergodic problem 
 \be\label{MFGergo}
\left\{\begin{array}{cl}
(i)& \overline \lambda +H(x,D\overline \phi) =f(x,\overline m(x))\\
(ii) & -{\rm div} (\overline mD_pH(x, D\overline\phi))=0\\
(iii)& \overline m\geq 0, \; \int_{\T^d} \overline m= 1
\end{array}\right.
\ee
and to show that this problem is well-posed. 
 
 \subsection{Well-posedness of the ergodic problem}
 
 Let us start with the definition of  solution of the ergodic problem \eqref{MFGergo}. \medskip
 
\begin{definition}\label{def:weaksolergoMFG} We say that a triple $(\lambda,  \phi,m )\in \R \times W^{1, pr}(\T^d)\times L^q(\T^d)$ is a weak solution of \eqref{MFGergo} if
\begin{itemize}
\item[(i)] $m\geq 0$, $\ds \int_{\T^d}m=1$ and $mD_pH(x,D\phi)\in L^{1}(\T^d)$, 
\item[(ii)] Equation \eqref{MFGergo}-(i) holds in the following sense:
\be\label{eq:aeergo}
\ds \quad \lambda +H(x,D\phi(x))= f(x,m(x)) \quad \; \mbox{\rm a.e. in $\{m>0\}$}
\ee
 and 
\be\label{eq:distribergo}
\quad \lambda +H(x,D\phi(x))\leq  f(x,m) \quad \mbox{\rm a.e. in}\; \T^d,  
\ee 

\item[(iii)] Equation \eqref{MFGergo}-(ii) holds:  
\be\label{eqcontdefergo}
\ds \quad {\rm div}( mD_pH(x,D\phi))= 0\quad {\rm in }\; \T^d, 
\ee
in the sense of distribution.
\end{itemize}
\end{definition}

Note that, if $(\lambda, m, \phi)$ is a solution, then  $\phi$ is H\"older continuous because, by \eqref{eq:conqr}, $pr>d$. \medskip

\begin{theorem}\label{theo:mainexergo} There exists at least one  solution $(\lambda,\phi,m)$ to the ergodic MFG system \eqref{MFGergo}. Moreover, the pair $(\lambda,m)$ is unique. 
\end{theorem} \medskip

In the following  example, given in \cite{lions07},  the solution can be computed explicitly: \\
\noindent {\bf An example:} Let us assume that $H(x,p)=\frac12|p|^2-V(x)$ and $f(x,0)=0$ for any $x\in \T^d$. Let 
$$
\xi(x,\lambda)= \left\{\begin{array}{ll}
f^{-1}(x,\lambda- V(x)) & {\rm if }\; \lambda\geq V(x)\\
0 & {\rm otherwise}
\end{array}\right.
$$
where $f^{-1}$ is the inverse of $f$ with respect to the second variable. As $f(x,0)=0$ and $f$ is strictly increasing and coercive with respect to the second variable, $\xi$ is nonnegative  and continuous in all variables, strictly increasing and coercive with respect to the second variable.  Moreover, $\xi(x,-s)=0$ for $s\geq \|V\|_\infty$. In particular, there is a unique real number $\rho$ such that $\int_{\T^d}  \xi(x,\rho)dx = 1$. As a consequence, if $( \bar \lambda, \bar \phi, \bar m)$ is a solution to \eqref{MFGergo}, then 
$$
\bar \lambda = \rho, \; 
\bar m(x)= \xi(x, \bar \lambda), \; D\bar \phi=0.
$$

\subsection{Proof of Theorem \ref{theo:mainexergo}}

The  basic strategy of proof is the same as (but much simpler than) for the time-dependent problem: we introduce two optimal control problems, which turn out to be in duality. The optimality conditions for these problems are precisely given by the ergodic MFG system. The main difference is the role played by the variable $\lambda$.

\subsubsection{Two optimization problems in duality}

We define, on $\overline{ \s{K}}_0:= \R\times W^{1, pr}(\T^d)$, the functional 
\be\label{DefmathcalA}
{\mathcal A}(\lambda,\phi)= \int_{\T^d}  F^*\left(x,\lambda+H(x,D\phi(x)) \right)\ dx - \lambda. 
\ee
Our first optimization problem is 
\be\label{PB:dual2ergo}
\inf_{(\lambda,\phi)\in \overline{ \s{K}}_0} \mathcal A(\lambda,\phi)
\ee

In view of the coercivity assumptions on $F^*$ and $H$ given in \eqref{eq:hamiltonian_bounds} and \eqref{eq:cost_growth_star}, it is clear that problem \eqref{PB:dual2ergo} has at least one solution $(\lambda,\phi)$.

To describe the second optimization problem, let us denote by $\overline{\s{K}}_1$ the set of pairs $(m, w)\in L^1(\T^d) \times L^1(\T^d;\R^d)$ such that $m(x)\geq 0$ a.e., $\ds \int_{\T^d}m(x)dx=1$, and which satisfy in the sense of distributions
\be\label{conteqergo}
{\rm div} (w)=0\; {\rm in}\;  \T^d. 
\ee
We define on $\overline{\s{K}}_1$ the functional
$$
{\mathcal B}(m,w)= \int_{\T^d} m(x) H^*\left(x, -\frac{w(x)}{m(x)}\right)+ F(x,m(x)) \ dx ,
$$
where we use the same convention as in \eqref{conventionH*}. 
Since $H^*$ and $F$ are bounded below and $m\geq 0$ a.e., the integral in ${\mathcal B}(m,w)$  is well defined in  $\R\cup\{+\infty\}$. 
The second optimal control problem is the following: 
\be\label{Pb:mw2ergo}
\inf_{(m,w)\in \overline{\s{K}}_1} \mathcal B(m,w)\;.
\ee

The following statement says that the two problems are in duality: \medskip

\begin{lemma}\label{Lem:dualiteergo} We have
\be\label{minmax}
\min_{(\lambda,\phi)\in \overline{ \s{K}}_0}{\mathcal A}(\lambda,\phi) = - \min_{(m,w)\in \overline{\s{K}}_1} {\mathcal B}(m,w), 
\ee
Moreover, the minimum in the right-hand side is achieved by a unique pair $(m,w)\in \overline{\s{K}}_1$ satisfying $(m,w)\in  L^q( \T^d)\times L^{\frac{r'q}{r'+q-1}}( \T^d)$. 
\end{lemma} \medskip

\begin{remark}{\rm Note that $\frac{r'q}{r'+q-1}>1$ because $r'>1$ and $q>1$. 
}\end{remark}

\begin{proof} Let us start with the regularity and the uniqueness of the solution of \eqref{Pb:mw2ergo}. Let $(m, w)\in \overline{\s{K}}_1$ be optimal in the above system. From the growth conditions \eqref{eq:hamiltonian_conjugate_bounds} and \eqref{eq:cost_growth}, we have 
$$
 C \; \geq \; \ds  \int_{\T^d}  F(x,m ) +m H^*( x, -\frac{w}{m})\ dx  \;   \geq \; \ds  \int_{\T^d}  \left(\frac{1}{ C}|m|^{q}+
 \frac{m}{ C } \left|\frac{w}{m}\right|^{r'} - C \right) dx.
 $$
In particular, $m\in L^q$.  By H\"older inequality, we also have 
$$
\int_{\T^d} |w|^{\frac{r'q}{r'+q-1}} = \int_{\{m>0\}} |w|^{\frac{r'q}{r'+q-1}} \leq \|m\|_q^{\frac{r'-1}{r'+q-1}} \left(\int_{\{m>0\}} \frac{|w|^{r'}}{m^{r'-1}} \right)^{\frac{q}{r'+q-1}} \leq C,
$$
so that $w\in L^{\frac{r'q}{r'+q-1}}$. Finally, we note that there is a unique minimizer to \eqref{Pb:mw2ergo}, because the set $\overline{\s{K}}_1$ is convex and the maps $F(x,\cdot)$ and $H^*(x,\cdot)$  are strictly convex: thus $m$ is unique and so is $\ds \frac{w}{m}$ in $\{m>0\}$. As $w=0$ in $\{m=0\}$, uniqueness of $w$ follows as well. 

Next we prove equality \eqref{minmax}. Let us rewrite problem \eqref{Pb:mw2ergo} as a min-max problem: 
$$
 \begin{array}{l}
 \ds \inf_{(m,w)\in \overline{\s{K}}_1} \mathcal B(m,w) \;  = \\ 
 \qquad \qquad \ds 
\inf_{(m,w)\in L^q\times L^{\frac{r'q}{r'+q-1}}}\sup_{(\lambda,\phi)\in \R\times C^1(\T^d)} \int_{\T^d}
\left(m H^*\left(x, -\frac{w}{m}\right)+ F(x,m)+\lg D\phi,w\rg-\lambda m \right) dx + \lambda.
\end{array}
$$
In view of the convexity and coercivity properties of $H^*$ and $F$, we can use the min-max theorem  (cf. e.g., \cite{ekeland1976convex}) to get:
$$
 \begin{array}{l}
 \ds \inf_{(m,w)\in \overline{\s{K}}_1} \mathcal B(m,w) \;  \\ 
\qquad =  \ds 
\sup_{(\lambda,\phi)\in \R\times C^1(\T^d)} \inf_{(m,w)\in L^q\times L^{\frac{r'q}{r'+q-1}}} \int_{\T^d}
\left(m H^*\left(x, -\frac{w}{m}\right)+ F(x,m)+\lg D\phi,w\rg-\lambda m \right) dx + \lambda \\
\qquad =  \ds 
\sup_{(\lambda,\phi)\in \R\times C^1}  \int_{\T^d} \inf_{(m,w)\in \R\times \R^d}
\left(m H^*\left(x, -\frac{w}{m}\right)+ F(x,m)+\lg D\phi,w\rg-\lambda m \right) dx + \lambda.
\end{array}
$$
An easy computation shows that 
$$
\inf_{(m,w)\in \R\times \R^d}
\left(m H^*\left(x, -\frac{w}{m}\right)+ F(x,m)+\lg D\phi,w\rg-\lambda m \right)
= - F^*(x, \lambda+H(x,D\phi))
$$
so that 
$$
 \begin{array}{l}
 \ds \inf_{(m,w)\in \overline{\s{K}}_1} \mathcal B(m,w) \;  = \;  \ds -
\inf_{(\lambda,\phi)\in \R\times C^1} \left(  \int_{\T^d}F^*(x, \lambda+H(x,D\phi)) dx - \lambda\right).
\end{array}
$$
This proves the claim since the natural relaxation of the minimization problem in the right-hand side is \eqref{PB:dual2ergo} by coercivity of $F^*$. 
\end{proof}

\subsubsection{Existence of solutions of the ergodic MFG problem}

We now show the existence part of Theorem \ref{theo:mainexergo}.
The proof is based on the one-to-one correspondence between solutions of the ergodic MFG system and the two optimization problems \eqref{PB:dual2ergo} and \eqref{Pb:mw2ergo}. \medskip

\begin{proposition}\label{theo:mainergo}
If $(m,w)\in \overline{\s{K}}_1$ is a minimizer of \eqref{Pb:mw2ergo} and $(\lambda, \phi)\in \overline{ \s{K}}_0$ is a minimizer of \eqref{PB:dual2ergo}, then $(m,\phi)$ is a solution of the mean field game system \eqref{MFGergo} and $w= -mD_pH(\cdot,D\phi)$ a.e.. 

Conversely, any  solution $(\lambda,\phi, m)$ of \eqref{MFGergo}  is such that the pair $(m,-mD_pH(\cdot,D\phi))$ is the minimizer of \eqref{Pb:mw2ergo} while $(\phi, f(\cdot,m))$ is a minimizer of \eqref{DefmathcalA}. 
\end{proposition}

\begin{proof} Let $(m,w)\in \overline{\s{K}}_1$ be a solution of \eqref{Pb:mw2ergo} and  let $(\lambda, \phi)\in {\mathcal K}_0$ be a solution of the  problem \eqref{Pb:mw2ergo}. Recall that $\ds mH^*(\cdot, -\frac{w}{m}) \in L^1$. From Lemma \ref{Lem:dualiteergo}, we have 
$$
0=  \int_{\T^d} m H^*(x, -\frac{w}{m})+ F(x,m)+ F^*(x,\lambda+H(x,D\phi))-\lambda\;.
$$
Since $m\in L^q$ while $H(\cdot,Du)\in L^p$,  we have by convexity of $F$, 
$$
\begin{array}{l}
\ds\int_{\T^d} m H^*(x, -\frac{w}{m})+ F(x,m)+ F^*(x,\lambda+H(x,D\phi))-\lambda \\
\qquad \qquad 
\ds\geq \; \int_{\T^d} m\left( H^*(x, -\frac{w}{m})+ \lambda+H(x,D\phi)\right)-\lambda.
\end{array}
$$
We now use the fact that $\int m=1$ and that  $D\phi\in L^{pr}$ while $w\in L^{(r'q)/(r'+q-1)}= (L^{pr})'$ to infer that 
$$
\begin{array}{l}
\ds \int_{\T^d} m\left( H^*(x, -\frac{w}{m})+ \lambda+H(x,D\phi)\right)-\lambda\geq 
\ds - \int_{\T^d} \lg w, D\phi\rg =0,
\end{array}
$$
where the last inequality comes from the constraint ${\rm div} (w)=0$. 
Since equality holds in the above string of inequalities, one must have 
$$
F(x,m)+ F^*(x,\lambda+H(x,D\phi)) =m(x)(\lambda+H(x,D\phi)) \qquad {\rm a.e.,}
$$
and 
$$
m(x)\left( H^*(x, -\frac{w(x)}{m(x)})+H(x,D\phi(x))\right)= -\lg w(x),D\phi(x)\rg \qquad {\rm a.e..}
$$
The first equality implies that $\lambda+H(x,D\phi(x)))=f(x,m(x))$ for almost every $x$ such that $m(x)>0$ since $F(x,\cdot)$ is strictly convex and smooth on $(0,+\infty)$. On $\{m=0\}$, one has $\lambda+H(x,Du(x)))\in \partial_mF(x,0)$. So, if $f(x,0)=-\infty$, then $m>0$ a.e. because $\lambda+H(x,Du(x)))$ is integrable. If $f(x,0)>-\infty$, then $\lambda+H(x,Du(x)))\leq f(x,0)$ in $\{m=0\}$ because $\partial_mF(x,0)= (-\infty, f(x,0)]$. So  \eqref{eq:distribergo} holds. 
The second inequality entails that $w(x)=-m(x)D_pH(x,D\phi(x))$ a.e. In particular, $mD_pH(\cdot,D\phi)\in L^{\frac{r'q}{r'+q-1}}(\T^d)$. So $(\lambda, \phi, m)$ is a solution to the ergodic problem. \\

Let us now assume that $(m,\phi)$ is a solution of  \eqref{MFGergo} and set $w=-mD_pH(x,D\phi)$. Then $(m,w)$ belongs to $\overline{\s{K}}_1$ and $(\lambda,\phi)\in \overline{ \s{K}}_0$. We first prove that $(m,w)$ is optimal for \eqref{Pb:mw2ergo}.  Let $(m',w')\in \overline{\s{K}}_1$ be an admissible pair. Without loss of generality we can assume that $m'H^*(x,-\frac{w'}{m'})\in L^1$ and $m'\in L^q$, because otherwise ${\mathcal B}(m',w')=+\infty$. Then, by convexity of $F$ with respect to the second variable, we have: 
$$
\begin{array}{rl}
\ds {\mathcal B}(m',w')\; =& \ds \int_{\T^d} m' H^*(x, -\frac{w'}{m'})+ F(x,m')\\
\geq &\ds \ds \int_{\T^d} m' H^*(x, -\frac{w'}{m'})+ F(x,m) +  f(x,m)(m'-m).
 \end{array}
 $$
By \eqref{eq:distribergo} and the fact that $\ds \int_{\T^d} (m'-m)=0$:
$$
\begin{array}{rl}
\ds {\mathcal B}(m',w') \;  \geq &  \ds \int_{\T^d} m' H^*(x, -\frac{w'}{m'})+ F(x,m) + (\lambda+H(x,D\phi))(m'-m) \\
 \geq & 
\ds \int_{\T^d} m' ( H^*(x, -\frac{w'}{m'})+H(x,D\phi)) + F(x,m)-H(x,D\phi)m  \\
 \geq & \ds \int_{\T^d}  -\lg w',D\phi\rg + F(x,m)-H(x,D\phi)m.
 \end{array}
 $$
Recall that ${\rm div} (w')=0$, so that $\ds \ds \int_{\T^d}  \lg w',D\phi\rg=0$. Moreover,  equality $w=-mD_pH(x,D\phi)$ implies that $H(x,D\phi)= \lg D_pH(x,D\phi),D\phi\rg-H^*(x,-w/m)$. So we finally  obtain 
$$
\begin{array}{rl}
\ds \int_{\T^d}  -\lg w',D\phi\rg + F(x,m)-H(x,D\phi)m\; =  & \ds  \int_{\T^d} F(x,m)-m \lg D_pH(x,D\phi),D\phi\rg+mH^*(x,-w/m)\\
= & \ds \int_{\T^d} F(x,m)+mH^*(x,-w/m),
 \end{array}
 $$
because $m D_pH(x,D\phi)=-w$ and  ${\rm div} (w)=0$. This proves the optimality of $(m,w)$. \\
 
The argument for proving the optimality of $(\lambda,\phi)$ is similar: let $(\lambda',\phi')\in {\mathcal K}_0$ be another admissible pair.  Then, since $m\in \partial_\alpha   F^*(x,\lambda+H(x,D\phi))$ a.e. because of \eqref{eq:aeergo} and \eqref{eq:distribergo}, 
$$
\begin{array}{rl}
\ds \mathcal A(\phi',\lambda')\; = & \ds \int_{\T^d}  F^*(x,\lambda'+H(x,D\phi'))-\lambda'  \\ \geq &
\ds \int_{\T^d}  F^*(x,\lambda+H(x,D\phi)) + m(\lambda'-\lambda+ \lg D_pH(x,D\phi), D(\phi'-\phi)\rg) -\lambda'.
 \end{array}
$$
We use the fact that $\int_{\T^d}m=1$ while  $m D_pH(x,D\phi)=-w$ and  ${\rm div} (w)=0$ to infer immediately that 
$$
\int_{\T^d}   m(\lambda'-\lambda+ \lg D_pH(x,D\phi), D(\phi'-\phi)\rg)= \lambda'-\lambda,
$$
which shows the optimality of $(\lambda, \phi)$. 
\end{proof}

\subsubsection{Uniqueness of the solution of the MFG ergodic system}

We conclude the section by proving the uniqueness part of Theorem \eqref{theo:mainexergo}. Following Proposition \ref{theo:mainergo}, we know that, if $(\lambda,\phi, m)$ is a solution to \eqref{MFGergo}, then $(m,w)$ (where $w=-mD_pH(\cdot,D\phi)$) is a solution of the minimization problem \eqref{Pb:mw2ergo}. By strict convexity of this latter problem, the pair $(m,w)$ is unique. Since, as explained in the proof of Proposition \ref{theo:mainergo}, 
 $mH(x,D\phi)= -\lg D_pH(x,D\phi),w\rg-mH^*(x,-w/m)$, we have by \eqref{eq:aeergo}:
 $$
 \lambda m- \lg w, D\phi\rg -mH^*(x,-w/m)= mf(x,m) \; {\rm a.e.}, 
 $$
 which, after integration over $\T^d$ gives, since ${\rm div}(w)=0$, 
 $$
 \lambda= \int_{\T^d} mH^*(x,-w/m)- mf(x,m).
$$
In particular, $\lambda$ is unique because the pair $(m,w)$ is unique.


\section{Asymptotic behavior}\label{sec:asymp}

In this section, we consider the asymptotic average of the solution of the finite time horizon mean field game system as time tends to infinity. For this we fix a terminal condition $\phi_f:\bb{T}^d\to \R$  for the mean field game system. We assume that $\phi_f$ is $C^1$ on $\bb{T}^d$. 

Let $(\phi^T, m^T)$ be a ``good" solution of the finite time horizon mean field game system on $(0,T)\times \T^d$
\begin{equation} \label{MFG}
\left\{\begin{array}{l}-\partial_t\phi + H(x,D\phi) = f(x,m)  \\
\partial_t m - \mathrm{div} \left(mD_p H(x,D\phi)\right) = 0 \\
\phi(T,x) = \phi_f(x), m(0,x) = m_0(x).
\end{array}\right.
\end{equation}
and $(\bar \lambda, \bar \phi,\bar m)$ be a solution of the ergodic mean field game system \eqref{MFGergo}. Let us define the rescaled functions 
$$
\psi^T(s,x)= \phi^T(sT,x), \qquad \mu^T(s,x)= m^T(sT,x)\qquad \forall (s,x)\in (0,1)\times \T^d. 
$$

Our main result is the following. \medskip
\begin{theorem}\label{thm:CvTheo} As $T\to+\infty$, 
\begin{itemize}
\item $(\mu^T)$ converges to $\bar m$ in $L^\theta((0,1)\times \T^d)$ for any $\theta\in [1,q)$, 
\item $\psi^T/T$ converges to the map $s\to \bar \lambda(1-s)$ in $L^\theta((\delta,1)\times \T^d)$ for any $\theta\geq 1$ and any $\delta\in (0,1)$. 
\end{itemize}
\end{theorem}

The proof is based, for the convergence for $\mu^T$, on the energy estimate given in Proposition \ref{prop:energie} and, for the convergence for $\psi^T/T$, on a passage to the limit in the optimization problem \eqref{pr:relaxed}.

\subsection{Energy estimates and the convergence of $m^T$}

Let us first provide various estimates on the solution of the time dependent system: \medskip
\begin{lemma}\label{lem:phiTbd} There is a constant $C$ such that, for any $T\geq 1$,  
$$
\int_0^T\int_{\T^d} (m^T(t,x))^q+(\alpha^T(t,x))_+^p+ |D\phi^T(t,x)|^r\ dxdt+ \|\phi^T\|_{L^\infty([1,T]\times \T^d)} \leq CT,
$$
where $\alpha^T(t,x)=f(x,m^T(t,x))$. 
\end{lemma}

\begin{proof} Let $w^T := -m^T D_pH(x,D\phi^T)$ and $\alpha^T:= f(x,m^T)$.
Using $(m,w):=(m_0, 0)\in {\mathcal K}_1$ as a competitor in the dual problem \eqref{pr:dual}, we see that, as $(m^T,w^T)$ is an optimum,
$$
\begin{array}{l}
\ds  \int_0^T\int_{\T^d} m^T(t,x) H^*\left(x, -\frac{w^T(t,x)}{m^T(t,x)}\right)+ F(x,m^T(t,x)) \ dxdt + \int_{\T^d} \phi_f(x)m^T(T,x)dx \\
\qquad \qquad \ds \leq  \int_0^T\int_{\T^d} m_0(x) H^*(x, 0)+ F(x,m_0(x)) \ dxdt + \int_{\T^d} \phi_f(x)m_0(x)dx\; \leq \; CT.
\end{array}
$$
As $\phi_f$ is bounded, $H^*$ satisfies (\ref{eq:hamiltonian_conjugate_bounds}), and $F$ satisfies \eqref{eq:cost_growth}, we infer the bound on $m^T$. 

Using the pair $(\phi,\alpha) := (\phi_f,H(\cdot,D\phi_f))\in \s{K}$ as a competitor in the relaxed problem \eqref{pr:relaxed}, we have, by optimality of $(\phi^T,\alpha^T)$, 
$$
\int_0^T\int_{\bb{T}^d} F^*(x,\alpha^T)-\int_{\bb{T}^d} m_0\phi^T(0) \leq 
\int_0^T\int_{\bb{T}^d} F^*(x,H(\cdot,D\phi_f))-\int_{\bb{T}^d} m_0\phi_f \leq CT.
$$
Fix $\ep>0$ to be chosen later and set $K_\ep= \sup_{\bb{T}^d} F(\cdot,\ep)$. Then $F^*(x,a)\geq \ep a -K_\ep$. So, in view of the growth of $F^*$, the above inequality implies
\be\label{iykdczjhd}
\int_0^T\int_{\bb{T}^d}  \frac{1}{C} |\alpha^T_+|^p -\ep |\alpha^T_-|-\int_{\bb{T}^d} m_0(\phi^T(0))_+ + \int_{\bb{T}^d} m_0(\phi^T(0))_-\leq CT.
\ee
Since $\phi^T$ satisfies the HJ inequality: $-\partial_t \phi^T+H(x,D\phi^T)\leq \alpha^T$ with $\phi^T(T)\leq \phi_f$, we have, for almost every $x\in \bb{T}^d$, 
\be\label{lbznlegnl}
\phi^T(0,x)+\int_0^T H(x,D\phi^T)dt \leq \int_0^T \alpha^T(t,x)dt+ \phi_f(x).
\ee
We integrate in space, using the growth of $H$ and the bound of $\phi_f$, and get
$$
\frac{1}{C}\int_0^T\int_{\bb{T}^d} |D\phi^T|^r\leq \int_0^T\int_{\bb{T}^d} \alpha^T- \int_{\T^d}\phi^T(0)+CT. 
$$
As  $m_0\geq 1/C_0$,  this inequality can be rewritten as: 
\be\label{kabelmkjn}
\frac{1}{C}\int_0^T\int_{\bb{T}^d} |D\phi^T|^r \leq \int_0^T\int_{\bb{T}^d} \alpha^T_+ -\int_0^T\int_{\bb{T}^d} \alpha^T_-+ C_0\int_{\T^d}m_0(\phi^T(0))_-+CT. 
\ee
Coming back to \eqref{lbznlegnl} that we integrate over $\phi^T(0)\geq 0$, we have from the lower bound on $H$, the regularity of $\phi_f$ and the assumption $m_0\leq C_0$, 
\be\label{kbsdflknln}
\frac{1}{C_0} \int_{\bb{T}^d} m_0(\phi^T(0))_+ \leq \int_{\bb{T}^d} (\phi^T(0))_+ \leq \int_0^T\int_{\bb{T}^d} \alpha^T_++CT.
\ee
 Putting together \eqref{iykdczjhd} and the above inequality: 
 $$
\int_0^T\int_{\bb{T}^d}  \frac{1}{C} |\alpha^T_+|^p -C_0 |\alpha^T_+| -\ep |\alpha^T_-| + \int_{\bb{T}^d} m_0(\phi^T(0))_-\leq CT.
$$
 Choosing $\ep=1/C_0$, we divide \eqref{kabelmkjn} by $C_0$ and add to the above inequality: 
 $$
\int_0^T\int_{\bb{T}^d} \frac{1}{C} \left(|D\phi^T|^r+  |\alpha^T_+|^p\right) -C |\alpha^T_+|  \leq  CT. 
 $$
This shows the desired bound on $D\phi^T$ and on $\alpha^T_+$. In view of \eqref{kbsdflknln}, these estimates imply that $m_0(\phi^T(0))_+$---and thus $(\phi^T(0))_+$---are bounded by $CT$ in $L^1$. Then we use \eqref{iykdczjhd} for $\ep>0$ small enough combined to \eqref{kabelmkjn} to get an $L^1$ bound for $m_0(\phi^T(0))_-$ and $(\phi^T(0))_-$. 

We now show the $L^\infty$ bound on $\phi^T$. Lemma \ref{lem:upper_bound} immediately implies that $\phi^T$ is bounded above by $CT$ 
(here as in the rest of the proof, we actually use the fact that we work with ``good" solutions, and we first apply the Lemma at the level of the regular approximating maps and then pass to the limit). For the lower bound, we first note that  for almost all $x,y\in \T^d$,
$$
-|\phi^T(0,x)|\leq \phi^T(0,x) \leq \phi^T(1,y) + C|x-y|^{r'} + C(\|\alpha^T_+\|_p+1). 
$$
We integrate this inequality over $x\in \bb{T}^d$ to obtain the uniform lower bound: $\phi^T(1,y)\geq -CT$.  Then using again Lemma \ref{lem:upper_bound}, we obtain, for almost all $x\in \T^d$, 
$$
-CT\leq  \phi^T(1,x) \leq  \phi^T(t,x) +C T,
$$
which shows the uniform estimate. 
\end{proof}

We are now ready to prove the convergence of $m^T$. 

\begin{proof}[Proof of Theorem \ref{thm:CvTheo}: convergence of $m^T$.] Let us fix $\beta\in (0,1)$, to be chosen later. Using Lemma \ref{lem:phiTbd} we find that there exists, for any $T>0$, two times $s_T\in [1,T^\beta]$ and $t_T\in [T-T^\beta, T]$ such that 
\be\label{khjsqbdfkjter}
\int_{\T^d} |D\phi^T(s_T,x)|^r+(m^T(s_T,x))^p+ |D\phi^T(t_T,x)|^r+(m^T(t_T,x))^p\ dx \leq C T^{1-\beta}
\ee
and for which the energy inequality given in Proposition \ref{prop:energie} and applied to the solutions $(\phi^T,m^T)$ and $(\bar \phi-\bar \lambda t, \bar m)$ holds: 
\be\label{ljkqbsdlgvj}
\frac{1}{T}\int_{s_T}^{t_T} \int_{\T^d} (m^T-\bar m)(f(m^T)-f(\bar m)) \leq \frac{1}{T}\left[\int_{\T^d} (m^T(t)-\overline m) (\phi^T(t)-\overline \phi)\ dx\right]_{s_T}^{t_T}.
\ee
where we have used  that $\int_{\T^d} \bar \lambda t (m^T(t)-\overline m)=0$ in the right-hand side of the equality. 
Let us set $\lg \phi^T(s_T)\rg=\int_{\bb{T}^d}\phi^T(s_T)$. At $t=s_T$, we have 
$$
\begin{array}{l}
\ds \frac{1}{T}\left|\int_{\T^d} (m^T(s_T)-\overline m) (\phi^T(s_T)-\overline \phi)\ dx\right| \\
\qquad \qquad = \ds \frac1T\left|\int_{\T^d} (m^T(s_T)-\overline m) (\phi^T(s_T)-\lg \phi^T(s_T)\rg-\overline \phi)\ dx\right| \\
\qquad \qquad \leq \ds \frac1T(\|m^T(s_T)\|_q+ \|\overline m\|_q) (\|\phi^T(s_T)-\lg \phi^T(s_T)\rg\|_p+ \|\bar \phi\|_p) .
\end{array}
$$
Using \eqref{khjsqbdfkjter}, the bound on $\|\phi^T\|_\infty$ in Lemma \ref{lem:phiTbd} and Poincar\'e inequality, we obtain
$$
\begin{array}{l}
\ds \frac1T (\|m^T(s_T)\|_q+ \|\overline m\|_q) (\|\phi^T(s_T)-\lg \phi^T(s_T)\rg\|_p+ \|\bar \phi\|_p) \\
\qquad \qquad \ds  \leq  \ds   CT^{\frac{1-\beta}{q}-1}  \left(\|\phi^T(s_T)-\lg \phi^T(s_T)\rg\|_\infty^{\frac{(p-r)_+}{p}}\|\phi^T(s_T)-\lg \phi^T(s_T)\rg\|_r+ C\right)\\
\qquad \qquad \ds   \leq   \ds   CT^{\frac{1-\beta}{q}-1} \left( T^{\frac{(p-r)_+}{p}}\|D\phi^T(s_T)\|_r+1\right) \;\leq \; CT^{\frac{1-\beta}{q}-1+\frac{(p-r)_+}{p}+\frac{1-\beta}{r}}
\end{array}
  $$
 We can argue in a similar way for the term $t=t_T$ to get
 $$
  \frac{1}{T}\left|\int_{\T^d} (m^T(t_T)-\overline m) (\phi^T(t_T)-\overline \phi)\ dx\right| \leq CT^{\frac{1-\beta}{q}-1+\frac{(p-r)_+}{p}+\frac{1-\beta}{r}}.
  $$ 
  Let us choose $\beta\in (0,1)$ such that $\ds \frac{1-\beta}{p}+\frac{1-\beta}{r}+\frac{(p-r)_+}{p}<1$. 
  Then \eqref{ljkqbsdlgvj} implies, after scaling and taking into account the above estimates:
  $$
  \limsup_{T\to+\infty}  \int_{T^{\beta-1}}^{1-T^{\beta-1}} \int_{\T^d} (\mu^T-\bar m)(f(x,\mu^T)-f(x,\bar m)) \ dxds =0.
  $$
We can then conclude the proof by using that $f$ is strictly increasing with respect to the second variable and that the $(\mu^T)$ are bounded in $L^q$. 
  \end{proof}

\subsection{Averaged limit of the variational problems and convergence of $\phi^T/T$} 

The convergence of $(\phi^T)$ is proved by letting $T\to \infty$ in the (averaged) optimization problem \eqref{pr:relaxed}. 

\begin{proof}[Proof of Theorem \ref{thm:CvTheo}: the convergence of $\phi^T/T$.]  Fix a subsequence $T\to+\infty$ (again denoted by $T$) and $s\in (0,1)$ such that $\mu^T(s)$ converges to $\overline m$ in $L^{\theta}(\T^d)$ for any $\theta\in [1, q)$. 

We first establish an estimate on the negative part $(\partial_t \phi^T)_-$ of the measure $\partial_t \phi^T$. Note that, as $- \partial_t \phi^T +H(x,D\phi^T) \leq \alpha^T(t,x)$, where $H(\cdot,D\phi^T)$ and $\alpha^T$ are integrable, $(\partial_t\phi^T)_-$ is absolutely continuous with respect to the Lebesgue measure, with 
$$
0\leq (\partial_t \phi^T)_- \leq \alpha^T(t,x)-H(x,D\phi^T) \leq  \alpha^T(t,x)+C\qquad {\rm a.e.},
$$
since $H$ is bounded below. We have therefore 
$$
\int_0^T\int_{\T^d} (\partial_t \phi^T)_-^p \leq C\int_0^T\int_{\T^d}(\alpha^T_+)^p + CT,
$$
so that,  by the bound on $\alpha^T$ given in Lemma \ref{lem:phiTbd}, 
\be\label{ineq:dspsi}
\int_0^T\int_{\T^d} (\partial_t \phi^T)_-^p\ dxdt \leq CT. 
\ee
Next we need to define a specific time $s_T$ near $sT$ at which $D\phi^T$ and $m^T$ are not too large. For this we use Lemma \ref{lem:phiTbd}, which says that 
$$
\int_0^T \int_{\T^d} |D\phi^T|^r+(m^T)^q\ dx \leq C T.
$$
Then, for any  $\beta\in (0,1)$ to be chosen below, there exists  a time $s_T\in [sT,sT+T^\beta]$ such that 
\be\label{khjsqbdfkjbis}
\int_{\T^d} |D\phi^T(s_T)|^r+(m^T(s_T,x))^q\ dx \leq C T^{1-\beta}.
\ee

By standard dynamic programming property the pair $(\phi^T, \alpha^T)$ is optimal for the (relaxed) problem of optimal control of HJ equation defined on the time horizon $[Ts,T]$ with initial measure $m^T(sT)$: 
$$
\begin{array}{l}
\ds  \inf_{(\phi,\alpha)\in {\mathcal K}} \int_{sT}^T\int_{\T^d}  F^*(x,\alpha(x,t))\ dxdt -\int_{\T^d} \phi(sT,x)m^T(sT,x)\ dx\\
\qquad \qquad \ds =  \int_{sT}^T\int_{\T^d}  F^*(x,\alpha^T(x,t))\ dxdt -\int_{\T^d} \phi^T(sT,x)m^T(sT,x)\ dx.
\end{array}
$$
Let $\gamma^T_s$ be the value of this problem divided by $T$. We claim that 
\be\label{lhbfidulh}
\gamma^T_s\leq (1-s)\gamma_\infty+ o(1),
\ee
where $\gamma_\infty$ is the value of the Problem \eqref{PB:dual2ergo} and  $o(1)\to0$ as $T\to+\infty$. Indeed, let us define 
$$
\phi(t,x)=\left\{\begin{array}{ll}
(T-t)\bar \phi(x)+ (t-T+1)\phi_f(x) & {\rm if }\; t\in [T-1,T]\\
\bar \phi(x)-\bar \lambda(t-T+1) & {\rm if } \; t\in [sT,T-1]
\end{array}\right.
$$
and
$$
\alpha(t,x)= 
\left\{\begin{array}{ll}
\left(\bar \phi(x)-\phi_f(x) + H(x, (T-t)D\bar \phi(x)+ (t-T+1)D\phi_f(x))\right)\vee 1 & {\rm if }\; t\in [T-1,T]\\
\bar \lambda + H(x,D\bar \phi(x))& {\rm if } \; t\in [sT,T-1]
\end{array}\right.
$$
Then the pair $(\phi,\alpha)$ belongs to ${\mathcal K}$ and, by definition of $\gamma^T_s$, 
$$
\gamma^T_s \leq \frac{1}{T} \int_{sT}^T\int_{\T^d} F^*(x,\alpha) - \frac{1}{T}\int_{\T^d} \phi(sT,x)m^T(sT,x)dx.
$$
As $\bar \phi$ and $\phi_f$ belong to $W^{1,rp}(\T^d)$, the right-hand side is bounded above by 
$$
(1-s) \int_{\T^d} F^*(x,\bar \lambda+H(x,\bar \phi))dx -(1-s)\bar \lambda+ \frac{C}{T}=(1-s)\gamma_\infty+ \frac{C}{T},
$$
since $(\bar \lambda,\bar \phi)$ is optimal in \eqref{DefmathcalA}. This implies \eqref{lhbfidulh}. In particular, if we set 
$$
\delta^T = \frac{1}{T}\int_{sT}^{s_T} \int_{\T^d}  F^*(x,\alpha^T(t,x))\ dxdt ,
$$
then 
\be\label{jhbfvoedln}
\begin{array}{l}
\ds (1-s)\gamma_\infty+ o(1) \\ 
\qquad \qquad \geq  \ds \frac1T\int_{s_T}^T\int_{\T^d}  F^*(x,\alpha^T(t,x))\ dxdt +\delta^T -\frac1T\int_{\T^d} \phi^T(sT,x)m^T(sT,x)\ dx \\
\qquad \qquad  \geq \ds (1-\frac{s_T}{T})\int_{\T^d}  F^*\left(x,\frac{1}{T-s_T}\int_{s_T}^T \alpha^T(t,x)\ dt\right)\ dx +\delta^T  -\frac1T\int_{\T^d} \phi^T(sT,x)\overline m(x)\ dx +o(1)
\end{array}
\ee
where the last inequality comes, for the first term, by convexity of $F^*$ and, for the last one, from the fact that $\|\phi^T\|_\infty$ is bounded by $CT$ on $[1,T]$ and $(m^T(sT))$ converges to $\overline m$ in $L^{1}$. 
Since $-\partial_t\phi^T+H(x,D\phi^T)\leq \alpha^T$ with $\phi^T(T,x)\leq\phi_f(x)$, we have for a.e. $x\in \T^d$, 
$$
\frac{1}{T-s_T}\int_{s_T}^T \alpha^T(t,x)\ dt \geq \frac{\phi^T(s_T,x)-\phi_f(x)}{T-s_T} +\frac{1}{T-s_T}\int_{s_T}^T H(x,D \phi^T)dt. 
$$
Let us now estimate $\delta^T$. We have, for any $\ep>0$ (recalling the inequality $F^*(x,a)\geq \ep a-K_\ep$),
$$
\begin{array}{rl}
\ds \delta^T \; \geq  & \ds   \frac{\ep}{T}\int_{sT}^{s_T} \int_{\T^d}  \alpha^T(t,x)\ dxdt -  \frac{K_\ep(s_T-sT)}{T})\\
\geq & \ds  \frac{\ep}{T}\left( \int_{\T^d} (\phi^T(sT,x)-\phi^T(s_T,x))dx + \int_{sT}^{s_T} \int_{\T^d}H(x,D\phi^T) dxdt\right) -K_\ep T^{\beta-1},
\end{array}
$$
where $\phi^T$ is bounded by $CT$ on $[1,T]$ and $H$ is bounded below. Thus
$$
\delta^T \geq -C\ep - (K_\ep+C\ep) T^{\beta-1},
$$
which implies that $\delta^T \geq o(1)$ as $T\to+\infty$. Collecting the above estimates and coming  back to \eqref{jhbfvoedln}, we obtain, as $F^*(x,\cdot)$ is nondecreasing,  
\be\label{kjhebflfn}
\begin{array}{l}
\ds(1-s)\gamma_\infty+o(1)\\ 
\qquad \geq  \ds (1-\frac{s_T}{T})
\int_{\T^d}  F^*\left(x,\frac{\phi^T(s_T,x)-\phi_f(x)}{T-s_T} +\frac{1}{T-s_T}\int_{s_T}^T H(x,D \phi^T)dt \right)\ dt\\
\qquad \qquad \ds  -\frac1T \int_{\T^d} \phi^T(sT,x)\bar m(x)\ dx.
\end{array}
\ee
We now estimate the last term: 
$$
\begin{array}{rl}
\ds -\frac1T \int_{\T^d} \phi^T(sT,x)\bar m(x)\ dx\;  = & \ds  -\frac1T \int_{\T^d} \phi^T(s_T,x)\bar m(x)\ dx +
\frac1T  \int_{sT}^{s_T} \int_{\T^d}  \partial_t \phi^T(s,x) \bar m(x)\ dxds \\
\geq &  \ds -\frac1T \lg \phi^T(s_T)\rg- \frac{\|\bar m\|_{q}}{T} \|\phi^T(s_T)-\lg \phi^T(s_T)\rg\|_p \\
& \ds \qquad  - \frac1T  \int_{sT}^{s_T} \int_{\T^d}  (\partial_t \phi^T(s,x))_- \bar m(x)\ dxds , 
\end{array}
$$
where, by the bound $\|\phi^T\|_\infty\leq CT$, Poincar\'e inequality and \eqref{khjsqbdfkjbis},  
$$
\begin{array}{rl}
\ds \frac{\|\bar m\|_{q}}{T} \|\phi^T(s_T)-\lg \phi^T(s_T)\rg\|_p \; \leq & \ds \frac{C}{T} \|\phi^T\|_\infty^{\frac{(p-r)_+}{p}}\|\phi^T(s_T)-\lg \phi^T(s_T)\rg\|_r\\ 
\leq & \ds CT^{\frac{(p-r)_+}{p}-1} \|D\phi^T(s_T)\|_r \; \leq\;  C T^{\frac{(p-r)_+}{p}+\frac{1-\beta}{r}-1},
\end{array}
$$
while, from H\"{o}lder inequality and \eqref{ineq:dspsi}:  
$$
\begin{array}{rl}
\ds \frac1T  \int_{sT}^{s_T} \int_{\T^d}  (\partial_t \phi^T(s,x))_-\bar m(x)\ dxds \;  \leq  & \ds \frac{\|\bar m\|_q}{T} (s_T-sT)^{1/q} \|  (\partial_t \phi^T)_-\|_p \leq C T^{\frac{\beta}{q}+\frac{1}{p}-1}.
\end{array} 
$$ 
So \eqref{kjhebflfn} becomes 
\be\label{MJLKJB}
\begin{array}{rl}
\ds(1-s)\gamma_\infty+o(1)\;  \geq   & \ds (1-\frac{s_T}{T})
\int_{\T^d}  F^*\left(x,-\frac{C}{T} +\frac{\phi^T(s_T,x)}{T-s_T}  +\frac{1}{T-s_T}\int_{s_T}^T H(x,D \phi^T)dt\right)\ dx\\
 & \qquad \qquad \ds   -\frac1T \lg \phi^T(s_T)\rg - C (T^{\frac{(p-r)_+}{p}+\frac{1-\beta}{r}-1}+T^{\frac{\beta}{q}+\frac{1}{p}-1}).
\end{array}
\ee
Note that, because of \eqref{khjsqbdfkjbis} and the $L^\infty$ bound on $\phi^T$, $\ds -\frac{1}{T-s_T} \phi^T(s_T,\cdot)$ converges, up to a subsequence, to a constant $\lambda$ in $L^r(\T^d)$.  In particular, $-\frac{1}{(1-s)T} \lg \phi^T(s_T)\rg$ converges also to $\lambda$. On the other hand, if we set 
$$
\overline \phi^T(x)= \frac{1}{T-s_T}\int_{s_T}^T (\phi^T(t,x)-\lg \phi^T(t)\rg) dt, 
$$
then we have by Lemma \ref{lem:phiTbd},
$$
\int_{\T^d} |D\overline \phi^T(x)|^rdx\leq C\qquad {\rm and }\qquad \int_{\T^d}\overline \phi^T(x)dx=0, 
$$
so that (again up to a subsequence), $\overline \phi^T$ converges weakly to a map $\phi$ in $W^{1,r}(\T^d)$. 
We rewrite \eqref{MJLKJB} by using the above notation and the  convexity of $H$:
$$
\begin{array}{rl}
\ds(1-s)\gamma_\infty+o(1)\;  \geq   & \ds (1-\frac{s_T}{T})
\int_{\T^d}  F^*\left(x,-\frac{C}{T} +\frac{\phi^T(s_T,x)}{T-s_T}  +H(x,D\overline \phi^T)\right)\ dx  \\
 & \qquad \qquad \ds -\frac1T \lg \phi^T(s_T)\rg - C (T^{\frac{(p-r)_+}{p}+\frac{1-\beta}{r}-1}+T^{\frac{\beta}{q}+\frac{1}{p}-1}).
\end{array}
$$
Letting $T\to+\infty$, we get by convexity of the map $(a,b)\to F^*(x, a+H(x,b))$ and choosing $\beta$ such that 
$\frac{(p-r)_+}{p}+\frac{1-\beta}{r}-1<0$, 
$$
\begin{array}{rl}
\ds(1-s)\gamma_\infty\;  \geq   & \ds (1-s)
\int_{\T^d}  F^*(x,\lambda +H(x,D\phi))\ dx -(1-s)\lambda. 
\end{array}
$$
Therefore the pair $(\lambda, \phi)$ is a solution to the optimization problem \eqref{PB:dual2ergo}. By uniqueness of the ergodic constant, we get $\lambda=\overline \lambda$. By definition of $\phi^T$ and \eqref{khjsqbdfkjbis}, we have therefore 
$$
\lim_{T\to+\infty} \frac{1}{T} \lg \phi^T(s_T,\cdot)\rg = -\overline \lambda. 
$$
Inequality $-\partial \phi^T+H(x,D\phi^T)\leq \alpha^T$ implies, since $H$ is bounded below: 
\be\label{jhgjqccjjx}
\lg \phi^T(s_T)\rg -\lg \phi^T(sT)\rg \leq C(s_T-sT) + (s_T-sT)^{1/q} \|\alpha^T_+\|_p\leq CT^{\beta/q+1/p} ,
\ee
where we have used the bound $\|\alpha^T_+\|_p\leq CT^{1/p}$ in the last inequality. 
This shows that 
$$
\liminf_{T\to+\infty} \frac{1}{T} \lg \phi^T(sT)\rg \geq \liminf_{T\to+\infty} \lg \phi^T(s_T)\rg -CT^{\beta/q+1/p-1} \geq 
 -(1-s)\overline \lambda.
$$
To obtain the opposite inequality, we just need to pick a time $s_T\in [sT-T^\beta,sT]$ for which \eqref{khjsqbdfkjbis} holds and prove as above  that 
$(\lg \phi^T(s_T)\rg/T)$ converges to $-(1-s)\overline \lambda$. Since, as for  \eqref{jhgjqccjjx},  we have 
$$
\lg \phi^T(sT)\rg -\lg \phi^T(s_T)\rg \leq CT^{\beta/q+1/p}, 
$$
we can then prove that 
$$
\limsup_{T\to+\infty} \frac{1}{T} \lg \phi^T(sT,\cdot)\rg \leq 
-(1-s)\overline \lambda.
$$
This shows the convergence of $(\lg \phi^T(\cdot)\rg/T)$ to $-(1-\cdot)\overline \lambda$  in $L^\theta(0,1)$ for any $\theta\geq 1$. Then,  by Poincar\'e inequality and the bound on $\|D\phi^T\|_r$,  
$$
\begin{array}{rl}
\ds \|\frac{\phi^T}{T}+(1-\cdot)\overline \lambda\|_{L^r((0,T)\times \T^d)}\; \leq & \ds \|\frac{\phi^T}{T}-\frac{\lg \phi^T\rg}{T}\|_r+\|\frac{\lg \phi^T\rg}{T}+(1-\cdot)\overline \lambda\|_r \\
\leq & \ds CT^{-1/r'}+ \|\lg \phi^T\rg/T+(1-\cdot)\overline \lambda\|_r
\end{array}
$$
 where the right-hand side vanishes as $T\to+\infty$. As $(\phi^T/T)$ is bounded in $[1,T]$, we conclude that the map $(\psi^T/T)$ converges to the map $s\to -(1-s)\bar \lambda$ in $L^\theta((\delta,1)\times \T^d)$ for any $\theta\geq 1$ and any $\delta\in (0,1)$. 
\end{proof}


\bibliographystyle{amsalpha}

\newcommand{\etalchar}[1]{$^{#1}$}
\providecommand{\bysame}{\leavevmode\hbox to3em{\hrulefill}\thinspace}
\providecommand{\MR}{\relax\ifhmode\unskip\space\fi MR }
\providecommand{\MRhref}[2]{%
  \href{http://www.ams.org/mathscinet-getitem?mr=#1}{#2}
}
\providecommand{\href}[2]{#2}

\end{document}